\documentclass[preprint,12pt]{elsarticle}

\usepackage{amssymb}
\usepackage{amsmath}
\usepackage{amsthm}
\usepackage{mathrsfs}

\usepackage{tikz}
\usetikzlibrary{calc}

\newcommand{\K}{\mathbb{K}}
\newcommand{\F}{\mathbb{F}}

\newcommand{\Points}{\mathcal{P}}
\newcommand{\Lines}{\mathcal{L}}
\newcommand{\I}{\mathtt{I}}

\newcommand{\SCHex}{\mathsf{H}(q)}

\newcommand{\TT}{\mathsf{T}(q^3,q)}

\newcommand{\T}{\mathsf{T}}

\newcommand{\Qk}{\mathsf{Q}^+(7,\K)}
\newcommand{\Qq}{\mathsf{Q}^+(7,q^3)}
\newcommand{\Q}{\mathsf{Q}}

\newcommand{\PG}{\mathsf{PG}(7,q^3)}
\newcommand{\PGk}{\mathsf{PG}(7,\K)}

\newcommand{\Lt}{\mathcal{L}^{\mathsf{T}}}
\newcommand{\Pt}{\mathcal{P}^{\mathsf{T}}}

\newcommand{\polarity}{\theta}
\newcommand{\triality}{\tau}

\newtheorem{theorem}{Theorem}
\newtheorem{definition}{Definition}
\newtheorem{corollary}{Corollary}
\newtheorem{lemma}{Lemma}
\theoremstyle{remark}
\newtheorem{remark}{Remark}
\newtheorem{example}{Example}

\journal{Discrete Mathematics}

\begin{document}

\begin{frontmatter}

\title{Intersection numbers of the natural embedding of the twisted triality hexagon \(\TT\) in \(\PG\)}

 \author[UC]{Sebastian Petit\corref{cor1}}
 \ead{sebastian.petit@pg.canterbury.ac.nz}
 \cortext[cor1]{Corresponding author}
 
 \author[UC]{Geertrui Van de Voorde}
 \ead{geertrui.vandevoorde@canterbury.ac.nz}
 
 \affiliation[UC]{organization={School of Mathematics and Statistics, University of Canterbury},
             addressline={Private Bag 4800},
             city={Christchurch},
             postcode={8140},
             state={Canterbury},
             country={New Zealand}}
 
\begin{abstract}
	In this paper, we study and characterise the natural embedding of the twisted triality hexagon \(\TT\) in \(\PG\). We begin by describing the possible intersections of subspaces of \(\PG\) with \(\TT\). Then, we provide conditions on a set of lines \(\Lines\) which ensure that \(\Lines\) forms the line set of a naturally embedded twisted triality hexagon.
	This work follows up on similar results for the split Cayley hexagon by J. A. Thas and H. Van Maldeghem (2008) and F. Ihringer (2014).
\end{abstract}




\end{frontmatter}

\section{Introduction}	

\emph{Generalised hexagons}, and the overarching \emph{generalised polygons}, were introduced in 1959 by Jacques Tits \cite{Tits1959}. 
Since then, these highly symmetrical structures have been studied intensely for their intriguing properties.
The monograph \cite{VanMaldeghem1998} describes the theory of generalised polygons and the background in more detail.
In the finite case, Feit and Higman have shown that thick generalised \(n\)-gons only exists for \(n\in \{3,4,6,8\}\) \cite{FeitHigman1964}.
Almost by definition, generalised triangles are equivalent to \emph{axiomatic projective planes} and hence have an even longer history.
Similarly, generalised quadrangles have their own extensive literature.

We focus on generalised hexagons.
Up to duality, only two classes of finite thick generalised hexagons are known: the \emph{split Cayley hexagons} of order \((q,q)\) and the \emph{twisted triality hexagons} of order \((q^3,q)\). In \cite{ThasVanMaldeghem2008}, Thas and Van Maldeghem characterised the natural embedding of the split Cayley hexagon in \(\text{PG}(6,q)\) using the intersection numbers with projective subspaces. Later, this was refined by Ihringer in \cite{Ihringer2014}.

This paper contributes to this study by characterising the natural embedding of the twisted triality hexagon \(\TT\) in \(\text{PG}(7,q^3)\). In accordance with \cite{ThasVanMaldeghem2008} and \cite{Ihringer2014} we focus on the set of lines of the generalised hexagon. 
After some introduction (Section \ref{sec:Background}), we first describe in detail (Section \ref{sec:EmbeddedTProp}) how subspaces of the projective space can intersect with the set of lines of a naturally embedded twisted triality hexagon. Then (Section \ref{sec:LineSet}), we start from a set of lines satisfying a number of these properties and show it to be the set of lines of a naturally embedded twisted triality hexagon. 
As a crucial tool, we use items (i) and (ii) from the main result of \cite{ThasVanMaldeghem1998}:
\begin{theorem}[Thas \& Van Maldeghem \cite{ThasVanMaldeghem1998}]\label{thm:KeyTheorem}
	\,
	\begin{enumerate}
		\item  If \(\Gamma\) is a thick generalized hexagon of order \((s,t)\) regularly lax embedded in \(\mathsf{PG}(d,p^h)\),
		then \(d \in \{5,6,7\}\), \(\Gamma\) is a classical generalized hexagon, and there exists a subgeometry
		\(\mathsf{PG}(d, s)\) over the subfield \(\text{GF}(s)\) of \(\text{GF}(p^h)\) such that \(\Gamma\) is naturally embedded in
		\(\mathsf{PG}(d, s)\).
		\item If a thick generalized hexagon \(\Gamma\) of order \((s, t)\) is flatly and fully embedded in \(\mathsf{PG}(d, s)\), then \(d\in \{4,5,6,7\}\) and \(t\le s\). Also, if \(d = 7\), then \(\Gamma\cong \T(s,\sqrt[3]{s})\) and the embedding is natural. If \(d = 6\) and \(t^5 > s^3\), then \(\Gamma\cong \mathsf{H}(s)\) and the embedding is natural. If \(d = 5\) and \(s=t\), then \(\Gamma\cong \mathsf{H}(s)\) with \(s\) even and the embedding is natural.
	\end{enumerate}	
\end{theorem}

\subsection*{Statement of the main results}

Let \(\Lines\) be a set of lines of \(\PG\).

\begin{definition}\label{def:LSupported}
	An \(n\)-dimensional subspace \(U\) of \(\PG\) is \emph{\(\Lines\)-supported} if the set of all lines of \(\Lines\) in \(U\) spans the space \(U\).
\end{definition}

We consider the following properties:
\begin{itemize}
	\item \emph{(Pt)} Every point of \(\PG\) is incident with \(0\) or \(q+1\) elements of \(\Lines\). 
	\item \emph{(Pl)} Every plane of \(\PG\) contains \(0\), \(1\) or \(q+1\) elements of \(\Lines\).
	\item \emph{(Sd)} Every solid of \(\PG\) contains \(0, 1, q+1\) or \(2q+1\) elements of \(\Lines\).
	\item \emph{(4d)} Every \(\Lines\)-supported \(4\)-space of \(\PG\) contains \(q^2+q+1\) or \(q^2+2q+1\) elements of \(\Lines\).
	\item \emph{(5d)} Every \(\Lines\)-supported \(5\)-space of \(\PG\) contains \(q^3+1\), \(q^3+q^2+q+1\), \(q^3+2q^2+2q+1\) or \(q^4+q+1\) elements of \(\Lines\).
	\item \emph{(6d)} Every \(\Lines\)-supported \(6\)-space of \(\PG\) contains \(q^5+q^4+q+1\) or \(q^5+q^4+q^3+q^2+q+1\) elements of \(\Lines\).
	\item \emph{(To)} \(|\Lines| \le q^9+q^8+q^5+q^4+q+1\).
\end{itemize}

We also consider the following variation:
\begin{itemize}
	\item \emph{(4d')} Every \(4\)-space of \(\PG\) contains at most \(q^2+2q+1\) elements of \(\Lines\).
\end{itemize}

\begin{theorem}\label{thm:Main1:SatisfiesProperties}
	The line set \(\Lines\) of a regularly embedded twisted triality hexagon \(\TT\) in \(\PG\) satisfies the properties (Pt), (Pl), (Sd), (4d), (4d'), (5d), (6d) and (To).
\end{theorem}

\begin{theorem}\label{thm:Main1.5:4dIsTwistedTriality}
	Let \(\Lines\) be a set of lines of \(\PG\). If \(\Lines\) satisfies the properties (Pt), (Pl), (Sd), (4d') and (To), then \(\Lines\) is the line set of a regularly embedded twisted triality hexagon \(\TT\) in \(\PG\).
\end{theorem}

In proving these results the following definitions will be useful (where we again assume that \(\Lines\) is a set of lines of \(\PG\)):
\begin{definition}\label{def:IsolatedAndIdeal}
	Let \(U\) be a subspace of \(\PG\).
	Given a point \(x\) of \(U\) on at least one line of \(\Lines\), we say that:
	\begin{itemize}
		\item \(x\) is \emph{\(\Lines\)-isolated} in \(U\) if no line of \(\Lines\) through \(x\) is contained in \(U\);
		\item \(x\) is \emph{\(\Lines\)-ideal} in \(U\) if all lines of \(\Lines\) through \(x\) are contained in \(U\).
	\end{itemize}
	If \(\Lines\) is clear from the context, we just write \emph{isolated} and \emph{ideal}.
\end{definition}

\begin{remark}
	Assume that we have the setting of Definition \ref{def:IsolatedAndIdeal}. If \(\Lines\) satisfies (Pt) and (Pl), then there are only three options for the point \(x\): either \(x\) is \(\Lines\)-isolated, \(x\) is \(\Lines\)-ideal or there is a unique line of \(\Lines\) through \(x\) in \(U\).
\end{remark}

\begin{remark}\label{rem:IsolatedAndIdealPoints}
	Whether or not a point is isolated (or ideal) depends on the subspace in which we consider the point. Every point on a line of \(\Lines\) is isolated in the subspace containing just that point and no point is isolated in the full projective space. It is trivial to see that if a point \(x\) is isolated in a subspace \(U\) then \(x\) is also isolated in all subspaces of \(U\) containing \(x\). Similarly, if \(x\) is ideal in a subspace \(U\) then \(x\) is also ideal in all subspaces which contain \(U\).
\end{remark}

\begin{lemma}\label{lem:IsolatedPointsInSandwichedSubspace}
	Let \(U, V\) and \(W\) be subspaces of \(\PG\) such that \(W\subseteq V \subseteq U\).
	If every line of \(\Lines\) in \(U\) intersects \(W\), then any isolated point of \(V\) is either isolated in \(U\) or isolated in \(W\) (for which it needs to be contained in \(W\)).
	
	In particular, if every line of \(\Lines\) in \(U\) intersects \(W\) and neither \(U\) nor \(W\) contains any isolated points, then \(V\) does not contain any isolated points.
\end{lemma}
\begin{proof}
	We may assume that \(x\) is isolated in \(V\) but not in \(U\). Let \(L\) be a line of \(\Lines\) in \(U\) through \(x\). Since \(L\) intersects \(W\), we can let \(y\) be a point on \(L\) in \(W\). If \(x\neq y\), then \(L\) would be contained in \(V\) contradicting that \(x\) was isolated in \(V\). Hence, \(x = y\) and contained in \(W\). As in Remark \ref{rem:IsolatedAndIdealPoints}, if \(x\) was not isolated in \(W\), then \(x\) would not be isolated in \(V\). Therefore, \(x\) must be isolated in \(W\).
\end{proof}

\section{Background}\label{sec:Background}

\subsection{Generalised hexagons}\label{GenBack}

In this section we highlight some of the basic properties and notation concerning generalised hexagons which will be used later. More details can be found in \cite{VanMaldeghem1998}.

\begin{definition}\label{def:WeakGeneralisedHexagon}
	A {\em weak generalised hexagon} is a point-line geometry \((\Points,\Lines,\I)\) satisfying:   
	\begin{enumerate}[(GH1)]
		\item There exists no \(k\)-gon (as a subgeometry) for \(2 \le k < 6\).
		\item Every two elements of \(\Points \cup \Lines\) are contained in a hexagon.
	\end{enumerate}
	A {\em thick generalised hexagon} is a weak generalised hexagon with the additional property:
	\begin{enumerate}[(GH3)]
		\item There exists a heptagon (as a subgeometry).
	\end{enumerate}
\end{definition}

\begin{remark}\label{rem:ThickGeneralisedHexagon}
	Property (GH3) is equivalent to the following, see for instance \cite[Lemma 1.3.2]{VanMaldeghem1998}:
	\textit{
		\begin{enumerate}[(GH3')]
			\item All points and lines are incident with at least 3 lines or points.
	\end{enumerate}}
\end{remark}

The {\em incidence graph} of a point-line geometry is the bipartite graph with as vertices the points and the lines and an edge between vertices \(x\) and \(y\) if and only if \(x\) and \(y\) are incident.
A weak generalised hexagon can also be defined using its incidence graph:

\begin{corollary}(See e.g. \cite[Theorem 1.5.10]{VanMaldeghem1998})\label{cor:incidenceGraph}
	Let \(\Gamma\) be a point-line geometry with at least two points and the properties that:
	\begin{itemize}
		\item every point is incident with at least two lines;
		\item every line is incident with at least two points.
	\end{itemize}
	Then, \(\Gamma\) is a weak generalised hexagon if and only if its incidence graph has diameter 6 and girth \(12\).
\end{corollary}

The {\em distance}, \(\delta(x,y)\), between two elements \(x\) and \(y\) of a weak generalised hexagon \(\Gamma\) is the distance inherited from its incidence graph. Two points, resp. lines, are {\em opposite} if they are at the furthest possible distance, i.e. 6.

The property that there are no quadrangles implies that for any two points \(x\) and \(y\) at distance 4 from each other, there is a unique point, which we will denote by \(x\bowtie y\), collinear with both.

The following useful observation follows directly from the definition of a weak generalised polygon. 

\begin{corollary}(See e.g. \cite[Theorem 1.3.5]{VanMaldeghem1998})\label{cor:UniqueClosest}
	Let \(x\) be a point and \(L\) be a line of a weak generalised hexagon \(\Gamma\). Then, there is a unique point on \(L\) closest (i.e., at smallest distance) to \(x\). 
\end{corollary}

The weak generalised hexagons that we consider have an order \((s,t)\) such that every line is incident with \(s+1\) points and every point is incident with \(t+1\) lines for some integers $s,t\geq 1$. 

\begin{remark}
	It is well-known (see for example \cite[Corollary 1.5.3]{VanMaldeghem1998}) that every thick finite generalised hexagon has an order. 
	Some restrictions on possible orders are known (\cite{FeitHigman1964} and \cite{HaemersRoos1981}).
\end{remark}

\begin{corollary}(See e.g. \cite[Lemma 1.5.4]{VanMaldeghem1998})\label{cor:Size}
	Let \(\Gamma\) be a finite weak generalised hexagon of order \((s,t)\). Then,
	\(|\Points| = (1+s)(1+st+s^2t^2)\) and \(|\Lines| = (1+t)(1+st+s^2t^2)\).
\end{corollary}

As mentioned before, up to duality, only two classes of finite thick generalised hexagons are known: the \emph{split Cayley hexagons} of order \((q,q)\) and the \emph{twisted triality hexagons} of order \((q^3,q)\). We focus on this second class although the split Cayley hexagons will still play an important role.

\begin{definition}\label{def:FullAndIdeal}
	A subgeometry \((\Points', \Lines',\I')\) of a weak generalised hexagon forming a weak generalised hexagon is called a \emph{subhexagon}. In general, a point of a subgeometry is called \emph{ideal} if all lines through that point are contained in the subgeometry. Dually, we call a line of a subgeometry \emph{full} if all points on that line are contained in the subgeometry. If all points are ideal we call the subgeometry itself \emph{ideal} and likewise we call a subgeometry \emph{full} if all of its lines are full.
\end{definition}

\begin{remark}\label{rem:IdealAndFull}
	Since we will only be working with weak generalised hexagons which have an order, we can restate the definitions of ideal and full as follows: 
	Let \(\Gamma\) be a weak generalised hexagon of order \((s,t)\) and let \(\Gamma'\) be a subhexagon of order \((s',t')\) of \(\Gamma\). Then \(\Gamma'\) is ideal precisely if \(t=t'\) and \(\Gamma'\) is full precisely when \(s=s'\).
\end{remark}

\begin{remark}\label{rem:SubhexagonsSCHandTT} 
	It is well known that the  split Cayley hexagon \(\SCHex\) contains ideal subhexagons of order \((1,q)\) \cite[Section 1.9.8]{VanMaldeghem1998}. We even have the following property: 
	any two opposite points of \(\SCHex\) are contained in a \((1,q)\) subhexagon.
	Likewise, the following is true in the twisted triality hexagon \(\TT\):
	any two opposite points are contained in an ideal split Cayley subhexagon.
\end{remark}

\begin{corollary}\label{cor:DistanceInSub}
	Let \(\Gamma\) be a generalised hexagon and let \(\Gamma'\) be a subhexagon of \(\Gamma\).
	Let \(x\) and \(y\) be two non opposite elements of \(\Gamma'\). Then, the shortest path between \(x\) and \(y\) is contained in \(\Gamma'\).
\end{corollary}

\begin{definition}\label{def:GammaSets}
	Let \(\Gamma = (\Points,\Lines,\I)\) be a generalised hexagon and let \(x\) be either a point or a line in \(\Gamma\). We define:
	\begin{align*}
		\Gamma_i(x) &:= \{y \in \Points\cup \Lines : \delta(x,y) = i\}, \\
		\Gamma_{\le i}(x) &:= \{y \in \Points\cup\Lines: \delta(x,y) \le i\} \\
	\end{align*}
	where \(\delta(x,y)\) is the distance between \(x\) and \(y\) inherited from the incidence graph.
\end{definition}

\begin{remark}\label{rem:NotationPL}
	For convenience we will use the sets \(\Points\) and \(\Lines\) as maps from sets and geometric structures to the subsets just containing points and lines, respectively, of the hexagon. 
	With this in mind, \(\Points(\Gamma_{\le i}(x))\) is the set of point of the hexagon at distance at most \(i\) from \(x\). 
	To simplify even further, we introduce the following abbreviations:
	\begin{align*}
		&\Points_{\le i}(x) := \Points(\Gamma_{\le i}(x)), 	&\Lines_{\le i}(x) := \Lines(\Gamma_{\le i}(x)), \\ 
		&\Points_{i}(x) := \Points(\Gamma_{i}(x)), 			&\Lines_{i}(x) := \Lines(\Gamma_{i}(x)).
	\end{align*}
	
\end{remark}

Let \(\Gamma\) be a weak generalised hexagon of order \((s,t)\).
Let \(L\) and \(M\) be two opposite lines of \(\Gamma\) and let \(x_L\) be an arbitrary point on \(L\). By Lemma \ref{cor:UniqueClosest}, there exists a unique point \(x_M\) on \(M\) at distance \(4\) from \(x_L\). The path we find from \(x_L\) to \(x_M\) goes through the point \(x := x_L\bowtie x_M\). Since \(x_L\) was arbitrary and there are no \(k\)-gons for \(k\in\{2,3,4,5\}\), we find \(s+1\) such points \textit{in between} \(L\) and \(M\).
We call this set, \(\Points_{3}(L) \cap \Points_{3}(M)\), the \emph{distance-3-trace},  \(T(L,M)\),  of \(L\) and \(M\).
Dually we define the \emph{distance-3-trace}, \(T(x,y)\), of two opposite points \(x\) and \(y\) to be the set \(\Lines_{3}(x) \cap \Lines_{3}(y)\).

\begin{definition}
	We call a line \(L\) \emph{distance-\(3\)-regular} if for any two lines \(M\) and \(N\) opposite \(L\), the distance-3-traces \(T(L,M)\) and \(T(L,N)\) are either the same or have at most one point in common.
	Dually, a point \(x\) is \emph{distance-\(3\)-regular} if for any two points \(y\) and \(z\) opposite \(x\), the distance-3-traces \(T(x,y)\) and \(T(x,z)\) are either the same or have at most one point in common.
	We say that a weak generalised hexagon \(\Gamma\) is \emph{distance-\(3\)-regular} if all points and all lines of \(\Gamma\) are distance-\(3\)-regular.
\end{definition}

It is easy to see from the construction that the following is true:
\begin{corollary}\label{cor:TraceOpposite}
	Let \(L\) and \(M\) be two opposite lines of a weak generalised hexagon \(\Gamma\). Then, all points of \(T(L,M)\) are mutually opposite.
	Dually, for any two opposite points \(x\) and \(y\) of a weak generalised hexagon \(\Gamma\), all lines of \(T(x,y)\) are mutually opposite.
\end{corollary}

\begin{remark}\label{rem:d3reg}
	It is well known that the twisted triality hexagon \(\TT\) is distance-\(3\)-regular (\cite{Ronan1980},\cite{VanMaldeghem1998}).
	This justifies the following definition:
\end{remark}

\begin{definition}\label{def:Regulus}
	Let \(L\) and \(M\) be two opposite lines of a twisted triality hexagon \(\TT\).
	
	We define \(R(L,M)\) to be the set of \(q+1\) lines in \(\bigcap_{x\in T(L,M)} \Lines_{3}(x)\) and call this the \emph{regulus} defined by \(L\) and \(M\).
\end{definition}

\begin{remark}
	For the natural embedding of \(\TT\) in \(\PG\), we will later see that \(R(L,M)\) is contained in \(\langle L, M \rangle\) and forms a subset of one generator set of a hyperbolic quadric \(\mathsf{Q}^+(3,q^3)\), but a priori this does not need to be true. 
\end{remark}

\begin{lemma}\label{lem:LineOppositeRegulus} 
	Let \(L\) and \(M\) be two opposite lines of a twisted triality hexagon \(\TT\). Let \(N\) be a line which is opposite all lines of \(R(L,M)\). Then, every point of \(T(L,M)\) is at distance 5 from \(N\).
\end{lemma}

\begin{lemma}\label{lem:MaximumSetOfOppositeLines}(See \cite{Offer2002} for the dual statement)
	Let \(\mathcal{S}\) be a set of mutually opposite lines in a weak generalised hexagon \(\Gamma = (\Lines,\Points,\I)\) of order \((s,t)\). Then, \(|\mathcal{S}|\le st^2+1\).
\end{lemma}

The following observation follows from a simple counting argument.
\begin{lemma}\label{lem:SpreadBlocksLines}
	Let \(\mathcal{S}\) be a set of \(q^3+1\) mutually opposite lines of a finite split Cayley hexagon \(\SCHex\). Then, every line of \(\SCHex\)  intersects a line in \(\mathcal{S}\).
\end{lemma}

\begin{lemma}\label{lem:IdealSubhexagonBlocksLines}
	Let \(\Gamma = (\Points,\Lines,\I)\) be a finite generalised hexagon \(\Gamma\) of order \((s,t)\). Let \(\Gamma' = (\Points',\Lines',\I')\) be an (ideal) subhexagon of order \((k,t)\). 
	If \(k^2t = s\), then every line of \(\Gamma\) intersects a line in \(\Gamma'\).
\end{lemma}
\begin{proof}
	For every line in \(\Gamma'\), there are \((s+1)t\) lines of \(\Gamma\) intersecting it. Note that only $k+1$ of those intersection points are points of $\Gamma'$. Hence, for every line \(L\) of \(\Gamma'\) there are \((s-k)t\) lines of \(\Gamma\) not in \(\Gamma'\) which intersect \(L\). Since each line of \(\Gamma\) intersecting two lines of \(\Gamma'\) is contained in \(\Gamma'\) (Corollary \ref{cor:DistanceInSub}), we can count all lines of \(\Gamma\) intersecting a line of \(\Gamma'\) as follows:
	\begin{align*}
		&|\Lines'|\cdot (s-k)t + |\Lines'| \\
		&= |\Lines'|\cdot((s-k)t+1) \\
		&= (t+1)(k^2t^2+kt+1)(st-kt+1) \\
		&= (t+1)(s^2t^2 + st + 1)= |\Lines|.\qedhere
	\end{align*} 
\end{proof}

\begin{corollary}\label{cor:IdealSubhexagonOfSCH}
	Let \(\Gamma'\) be a subhexagon of order \((1,q)\) of a split Cayley hexagon \(\SCHex\). Then, every line of \(\SCHex\) intersects a line of \(\Gamma'\).
\end{corollary}

\begin{corollary}\label{cor:SCHBlocksInTT}
	Let \(\SCHex\) be a split Cayley subhexagon of a twisted triality hexagon \(\TT\). Then, every line of \(\TT\) intersects a line of \(\SCHex\).
\end{corollary}

\subsection{The standard embedding of \(\TT\) in \(\PG\)}\label{EmbeddedT}

We now consider the standard embedding of a twisted triality hexagon \(\TT=:\T\) as described in for instance Section 2.4 in \emph{Generalized Polygons} \cite{VanMaldeghem1998}. We briefly repeat some relevant parts but refer to the book for details and a more in depth explanation.

We start from a hyperbolic quadric \(\Qk =: \Q\) in \(\PGk\). 

The points on \(\Q\) have homogeneous coordinates satisfying the following equation:
\[x_0 x_4 + x_1 x_5 + x_2 x_6 + x_3 x_7 = 0.\]

Since \(\Q\) is determined by a quadratic equation, we have the following property which will be useful later:
\begin{lemma}\label{lem:IntersectionLineQuadric}
	Let \(L\) be a line of \(\PGk\) then \(L\) contains \(0,1,2\) or \(|\K|+1\) points of \(\Q\).
\end{lemma}

The set of \(3\)-spaces in \(\Q\) is partitioned into two classes such that any two \(3\)-spaces from the same subclass meet in an odd dimensional subspace and any two \(3\)-spaces from different classes meet in an even dimensional subspace. Without loss of generality we call elements of one subclass \(3\)-spaces and elements of the second subclass \(3'\)-spaces.
We say that a \(3\)-space and a \(3'\)-space are incident if they meet in a plane.
Interestingly, both the \(3\)-spaces and the \(3'\)-space behave the same as the class of points and hence can be coordinatised in the same way.

We use the following trilinear form \(\mathcal{T}: V \times V \times V \to \K \), where \(V\) is an eight-dimensional vector space over \(\K\), to establish these coordinates:
\begin{align*}
	\mathcal{T}(x,y,z) = \,
	&\begin{vmatrix}
		x_0 &x_1 &x_2 \\
		y_0 &y_1 &y_2 \\
		z_0 &z_1 &z_2 
	\end{vmatrix}
	+
	\begin{vmatrix}
		x_4 &x_5 &x_6 \\
		y_4 &y_5 &y_6 \\
		z_4 &z_5 &z_6 
	\end{vmatrix} \\			
	&+x_3(z_0y_4 + z_1y_5+z_2y_6) + x_7(y_0z_4 + y_1z_5+y_2z_6) \\
	&+y_3(x_0z_4 + x_1z_5+x_2z_6) + y_7(z_0x_4 + z_1x_5+z_2x_6) \\
	&+z_3(y_0x_4 + y_1x_5+y_2x_6) + z_7(x_0y_4 + x_1y_5+x_2y_6) \\
	&-x_3y_3z_3-x_7y_7z_7. \\
\end{align*}

A point \(x\) now belongs to the \(3\)-space \(y\) if and only if \(\mathcal{T}(x,y,z')\) vanishes as a function of \(z'\). Similarly, a point \(x\) belongs to the \(3'\)-space \(z\) if and only if \(\mathcal{T}(x,y',z)\) vanishes as a function of \(y'\). A \(3\)-space \(y\) and a \(3'\)-space \(z\) are incident if and only if \(\mathcal{T}(x',y,z)\) is vanishes as a function of \(x'\).

The \(3\)-spaces and \(3'\)-spaces with their coordinates, will further be referred to as 1-points, \(\Points^{(1)}\), and 2-points, \(\Points^{(2)}\) respectively.
The original points will be referred to as \(0\)-points, \(\Points^{(0)}\).

\begin{example}\label{ex:Coordinate3Space}
	To find the \(3\)-space that corresponds to the \(1\)-point \\ \(y = (1,0,0,0,0,0,a,0)\), we substitute this into the trilinear form and require that the coefficients of all \(z_i\) vanish. This results in the following equations:
	\[\begin{cases}
		0 = x_2 \\
		0 = x_1 - a x_3 \\
		0 = x_4 \\
		0 = x_7 + a x_5. \\
	\end{cases}\]
\end{example}

We can now use the following triality which preserves incidence:
\begin{align*}
	\tau_\sigma: \Points^{(i)}&\to \Points^{(i+1)}: \\
	(x_0,x_1,x_2,x_3,x_4,x_5,x_6,x_7) &\to (x_0^\sigma,x_1^\sigma,x_2^\sigma,x_3^\sigma,x_4^\sigma,x_5^\sigma,x_6^\sigma,x_7^\sigma)
\end{align*}

where \(i = 0,1,2 \mod 3\) and \(\sigma\) is an automorphism of \(\K\) of order 1 or 3.

Let \(L\) be a line of \(\Q\). Any two points on \(L\) determine \(L\). Similarly we can see that \(L\) is defined by any two \(3\)-spaces through \(L\) and also by any two \(3'\)-spaces through \(L\).
Since \(\triality_\sigma\) preserves incidence, it is easy to see that the images of each such pair must determine the same line \(M\). Hence \(\triality_\sigma\) induces a map between all lines of \(\Q\).

With some abuse of notation, we see \(\triality_\sigma\) as a map on points, lines, \(3\)-spaces and \(3'\)-spaces of \(\Q\). Given a point \(x\) of \(\Q\), we write \(x^{\triality_\sigma}\) for the \(3\)-space corresponding to the image of \(x\) under \(\triality_\sigma\). Similarly, given a line \(L\) of \(\Q\), we write \(L^{\triality_\sigma}\) for the line which is its image under the map induced by \(\triality_\sigma\).

We call a point \(x\) of \(\Q\) \emph{absolute} with respect to \(\triality_\sigma\) if it is incident with its image under \(\triality_\sigma\), i.e. \(x\in x^{\triality_\sigma}\). We call a line of \(\Q\) absolute with respect to \(\triality_\sigma\) if it is fixed under the induced map from \(\triality_\sigma\), i.e. \(L = L^{\triality_\sigma}\). 

\begin{corollary}(See \cite[2.4.4]{VanMaldeghem1998})\label{cor:absoluteLine}
	Let \(x\) and \(x'\) be absolute points with respect to \(\triality_\sigma\). The point \(x'\) is contained in the \(3\)-space \(x^{\triality_\sigma}\) if and only if \(xx'\) is an absolute line with respect to \(\triality_\sigma\).
\end{corollary}

\begin{example}\label{ex:AbsoluteHexagon}
	Let \(e_i\) be the point with coordinates all zero except \(x_i = 1\). The triality \(\triality_\sigma\) maps \(e_0\) to the 3-space which has the following equations (see Example \ref{ex:Coordinate3Space} with \(a=0\)): 
	\begin{align}
		x_1=x_2=x_4=x_7=0. \label{test}
	\end{align}
	This shows that \(e_0\) is absolute since it satisfies equation (1). Similarly, one can see that \(e_1,e_2,e_4,e_5\) and \(e_6\) are absolute. Since \(e_5\) and \(e_6\) also satisfy equation (1), Corollary \ref{cor:absoluteLine} indicates that \(e_0e_5\) and \(e_0e_6\) are absolute lines.
	One can show that all six lines of the hexagon \(\mathscr{H} = (e_0,e_6,e_1,e_4,e_2,e_5)\) are absolute.
\end{example}

\begin{remark}\label{rem:absoluteLinesInAPlane}
	It follows from Corollary \ref{cor:absoluteLine} that all absolute lines through a point \(x\) are contained in \(x^{\triality_\sigma}\). By considering \(\triality_\sigma^2\), which has the same absolute points, we also find that all absolute lines through \(x\) are contained in \(x^{\triality_\sigma^2}\). This shows that all absolute lines through an absolute point are contained in a plane, namely \(x^{\triality_\sigma}\cap x^{\triality_\sigma^2}\).
\end{remark}

\begin{example}\label{ex:PlaneOfAbsoluteLines}
	Let \(x\) be the point \((1,0,0,0,0,0,a,0)\). The \(3\)-space \(x^{\triality_\sigma}\) has coordinates \((1,0,0,0,0,0,a^\sigma,0)\). As in Example \ref{ex:Coordinate3Space} we obtain the following equations for this \(3\)-space:
	\[\begin{cases}
		0 = x_2 \\
		0 = x_1 - a^\sigma x_3 \\
		0 = x_4 \\
		0 = x_7 + a^\sigma x_5. \\
	\end{cases}\]
	Combining these with the equations of \(x^{\triality_\sigma^2}\), we find that the plane \(x^{\triality_\sigma}\cap x^{\triality_\sigma^2}\), which contains all absolute lines through \(x\), has the following equations:
	\[\begin{cases}
		0 = x_2 \\
		0 = x_4 \\
		0 = x_1 - a^\sigma x_3 \\
		0 = x_1 + a^{\sigma^2} x_7 \\
		0 = x_7 + a^\sigma x_5 \\
		0 = x_3 - a^{\sigma^2} x_5. \\
	\end{cases}
	\]
	
\end{example}

We are now ready for the following theorem which was first proven by Tits in 1959 \cite{Tits1959}. See also \cite[Section 2.4]{VanMaldeghem1998}.
\begin{theorem}[Tits]\label{thm:GHfromTriality}
	The geometry \(\Gamma = (\Points_\text{abs},\Lines_\text{abs},\I)\) arising from the triality \(\triality_\sigma\) is a generalised hexagon of order \((|\K|, |\K^{(\sigma)}|)\), where \(\K^{(\sigma)}\) is the subfield of \(\K\) of those elements fixed by \(\sigma\).
\end{theorem}

\begin{definition}\label{def:SCandTT}
	By setting \(\sigma = 1\) we can produce a generalised hexagon over any finite field \(\F_q\) which is known as the \emph{split Cayley hexagon} and denoted by \(\SCHex\).
	
	For \(\sigma \neq 1\) we call the generalised hexagon obtained using theorem \ref{thm:GHfromTriality} a \emph{twisted triality hexagon}. Over a finite field, \(\K\) needs to be a Galois extension of degree 3 of \(\K^{(\sigma)}\) and we write \(\TT\) for the unique twisted triality hexagon over \(\F_{q^3}\).  
\end{definition}

\begin{remark}\label{rem:SplitCayleyHexagonInHyperplane}\label{rem:SplitCayleySubhexagon}
	Tits also proved in \cite{Tits1959} that all points and lines of the split Cayley hexagon constructed as above are contained in the hyperplane of \(\PG\) with equation \(x_3+x_7=0\).
	By restricting the coordinates of a twisted triality hexagon \(\TT\) to \(\F_q\), we find this same split Cayley hexagon as an ideal subhexagon.
\end{remark}

The described embedding of the generalised hexagons \(\SCHex\) and \(\TT\) is considered their \emph{natural} embedding.
From now on, we consider a twisted triality hexagon \(\TT =: \T\) which is naturally embedded in a quadric \(\Qq=:\Q\) in \(\PG\). To keep the reference to the generalised hexagon, we will write the points and lines of \(\T\) as \(\Pt\) and \(\Lt\) respectively.

\begin{remark}\label{rem:Full}
	Since \(\T\) has order \((q^3,q)\), for every line of \(\T\), all \(q^3+1\) points of \(\PG\) on that line are contained in \(\T\). This was precisely the definition of a full subgeometry (Definition \ref{def:FullAndIdeal}).
\end{remark}

In Remark \ref{rem:absoluteLinesInAPlane} we already observed the following property which we now formalise:
\begin{lemma}\label{lem:Flat}
	For every point \(x\in\Pt\) the set of collinear points in \(\Pt\) is contained in a plane of \(\PG\). 
\end{lemma}
We call any embedding which has this property \emph{flat}.

From now on we will use \(\polarity\) for the polarity on \(\PG\) generating \(\Q\) and \(\triality\) for the triality on \(\Q\) generating \(\T\) (We generally do not need the automorphism \(\sigma\) and therefore simplify \(\triality_\sigma\) to just \(\triality\)).

From the construction (see also Corollary \ref{cor:absoluteLine}) we observe the following:

\begin{corollary}\label{cor:trialityBasic}
	For any point \(x\) of \(\Pt\) we have:
	\begin{enumerate}
		\item\label{cor:trialityBasic:points} \(\Pt(x^\triality) = \Pt_{\le 2}(x).\)
		\item\label{cor:trialityBasic:lines} \(\Lt(x^\triality) = \Lt_{\le 1}(x).\)
		\item\label{cor:trialityBasic:noIso} The subspace \(x^\triality\) does not contain any isolated points.
		\item\label{cor:trialityBasic:noIsoPl} The subspace \(\langle \Pt_{\le 2}(x)\rangle\) does not contain any isolated points.
	\end{enumerate}
\end{corollary}

\begin{lemma}\label{lem:SolidsInQ}
	Any solid of \(\Q\) contains at most \(q+1\) elements of \(\Lt\).
\end{lemma}
\begin{proof}
	We prove that in any solid of \(\Q\) every two lines of \(\Lt\) intersect. The statement then follows from the fact that there are no triangles and at most \(q+1\) lines of \(\Lt\) through any given point.
	
	Assume that \(\Sigma\) is a solid of \(\Q\) which contains two lines \(L\) and \(M\) of \(\Lt\). Either \(\Sigma^\triality\) or \(\Sigma^{\triality^2}\) is a point of \(\Q\). Since \(L\) and \(M\) are absolute and \(\triality\) preserves incidence, we find that this point is on both \(L\) and \(M\).
\end{proof}

The construction and Lemma \ref{lem:IsolatedPointsInSandwichedSubspace} imply the following:
\begin{lemma}\label{lem:pointTangent}
	Given a point \(x\in\Pt\):
	\begin{enumerate}
		\item\label{lem:pointTangent:points} \(\Pt(x^\polarity) = \Pt_{\le 4}(x)\).
		\item\label{lem:pointTangent:lines}\(\Lt(x^\polarity) = \Lt_{\le 3}(x)\).
		\item\label{lem:pointTangent:noIso} The subspace \(x^\polarity\) contains no isolated points.
		\item\label{lem:pointTangent:noIsoSub} Any subspace of \(x^\polarity\) containing \(\Pt_{\le 2}(x)\), contains no isolated points.
	\end{enumerate}
\end{lemma}

Lemma \ref{lem:pointTangent}(\ref{lem:pointTangent:points}), shows that all points not opposite a given point are contained in a hyperplane of \(\PG\). In \cite{ThasVanMaldeghem1998}, the authors call an embedding \emph{weak} precisely when it has this property. In other literature, e.g. \cite{SteinbachVanMaldeghem2004}, an embedding is considered \emph{polarised} when for any point \(x\), the set of points not opposite \(x\) does not
generate \(\PG\). An embedding which is both flat (see Lemma \ref{lem:Flat}) and polarised is called \emph{regular}.

\begin{corollary}\label{cor:oppositeIffopposite}
	Two points \(x\) and \(y\) of \(\Pt\) are opposite in \(\T\) if and only if they are opposite in \(\Q\).
\end{corollary}

\begin{lemma}\label{lem:lineTangent}
	Given a line \(L\in\Lt\):
	\begin{enumerate}
		\item\label{lem:lineTangent:points} \(\Pt(L^\polarity) = \Pt_{\le 3}(L)\).
		\item\label{lem:lineTangent:lines} \(\Lt(L^\polarity) = \Lt_{\le 2}(L)\).
		\item\label{lem:lineTangent:noIso} The subspace \(L^\polarity\) contains no isolated points.
		\item\label{lem:lineTangent:noIsoSub} Any subspace of \(L^\polarity\) containing \(L\), contains no isolated points.
	\end{enumerate}
\end{lemma}

In Remark \ref{rem:SplitCayleySubhexagon}, we mentioned that in the twisted triality hexagon \(\TT\) we find the split Cayley subhexagon \(\SCHex\), contained in the hyperplane with equation \(x_3+x_7=0\) (see Remark \ref{rem:SplitCayleyHexagonInHyperplane}). 
One can see that all points of \(\T\) in this hyperplane are on a line of this split Cayley subhexagon.
The points of \(\T\) on two or more lines in this hyperplane are precisely those of this split Cayley subhexagon.
 
We say that the hyperplane with equation \(x_3+x_7 = 0\) \emph{intersects \(\T\) in a split Cayley hexagon}.
More generally, 
Let \(V\) be a hyperplane of \(\PG\). We say that \(V\) \emph{intersects \(\T\) in a split Cayley hexagon} if the lines \(\Lt(V)\) together with the points of \(\Pt(V)\) on at least two of these lines, form an ideal split Cayley subhexagon of \(\T\).

We will see that a similar situation can occur for \(5\)-dimensional subspaces of \(\PG\).
We say that a \(5\)-dimensional subspaces \(U\) \emph{intersects \(\T\) in a subhexagon of order \((1,q)\)} if the lines \(\Lt(V)\) together with the points of \(\Pt(V)\) on at least two of these lines, form a subhexagon of \(\T\) of order \((1,q)\).

\section{The standard embedding of \(\TT\) in \(\PG\) satisfies the properties}\label{sec:EmbeddedTProp}

We are now ready to investigate the properties of the standard embedding of \(\T = \TT\).
\begin{lemma}\label{lem:propPt}
	\textbf{(Pt)} Every point of \(\PG\) is incident with \(0\) or \(q+1\) lines of \(\Lt\). 
\end{lemma}
\begin{proof}
	Follows from \(\T\) having order \((q^3,q)\).
\end{proof}
\begin{lemma}\label{lem:pl}
	\textbf{(Pl)} Every plane of \(\PG\) is incident with \(0\), \(1\) or \(q+1\) lines of \(\Lt\).
\end{lemma}
\begin{proof}
	Let \(\pi\) be a plane of \(\PG\) containing two lines \(L\) and \(M\) of \(\Lt\). Since \(\T\) is a fully embedded subgeometry (see Remark \ref{rem:Full}) the intersection of \(L\) and \(M\) is a point of \(\Pt\). The statement follows from \(\T\) being flat (see Lemma \ref{lem:Flat}), property (Pt) (Lemma \ref{lem:propPt}) and \(\T\) not having triangles.
\end{proof}

\begin{lemma}
	\textbf{(To)} \(|\Lt| \le q^9+q^8+q^5+q^4+q+1\).
\end{lemma}
\begin{proof}
	Follows from \(\T\) having order \((q^3,q)\) and Corollary \ref{cor:Size}.
\end{proof}

We will use the following observation:
\begin{lemma}\label{lem:SupSubspace}
	Let \(V\) be an \(n\)-dimensional \(\Lt\)-supported (see Definition \ref{def:LSupported}). subspace of \(\PG\).
	Then, there exists a \(\Lt\)-supported subspace \(U\) of \(V\) with \(\dim(U) = n-1\) or \(\dim(U) = n-2\) and a line \(L\) of \(\Lt\) such that \(V = \langle U,L \rangle\). 
\end{lemma}
\begin{proof}
	Every line we remove from the set \(\Lt\) decreases the dimension of the space spanned by this set by either \(0,1\) or \(2\).
\end{proof}

In view of this, we can restate what we know about planes:

\begin{lemma}\label{lem:2supported}
	Let \(\pi\) be a \(2\)-dimensional \(\Lt\)-supported subspace of \(\PG\). 
	
	Then, \(\pi\) contains \(q+1\) lines of \(\Lt\) which all meet in a point \(x\) and, \(\pi\) does not contain isolated points.
\end{lemma}
\begin{proof}
	It follows from (Pl) (Lemma \ref{lem:pl}) and the definition of \(\Lt\)-supported that \(\pi\) contains \(q+1\) lines of \(\Lt\).
	We know from \(\T\) being flat and how we obtained property (Pl) that the \(q+1\) lines of \(\Lt\) in \(\pi\) all meet in a point \(x\).
	Lemma \ref{cor:trialityBasic} shows that \(\pi\) does not contain any isolated points.
\end{proof}


\begin{lemma}\label{lem:3supported}
	\textbf{(Sd)} Let \(\Sigma\) be a \(3\)-dimensional \(\Lt\)-supported subspace of \(\PG\). 
	
	Then, one of the following is true:
	\begin{itemize}
		\item \(|\Lt(\Sigma)| = 2q+1\). 
		
		More specifically, there is a line \(L\) in \(\Lt(\Sigma)\) and two points \(x\) and \(y\) on \(L\) such that there are \(q\) lines in \(\Lt(\Sigma)\) intersecting \(L\) in \(x\) and \(q\) lines in \(\Lt(\Sigma)\) intersecting \(L\) in \(y\).
		
		Hence, there are exactly two ideal points in \(\Sigma\) and they are collinear in \(\T\).
		
		\item \(|\Lt(\Sigma)| = q+1\). 
		
		More specifically, any two lines \(L\) and \(M\) in \(\Lt(\Sigma)\) are opposite in \(\T\) and \(R(L,M) = \Lt(\Sigma)\). 
		
		It follows that there are no ideal points in \(\Sigma\).
		
	\end{itemize} 
	In either case there are no isolated points.
\end{lemma}
\begin{proof}
	We consider two cases:  
	\begin{enumerate}
		\item \(\Sigma\) contains an \(\Lt\)-supported plane \(\pi\) and a line \(L\) of \(\Lt\) intersecting \(\pi\) in a point.
		\item \(\Sigma\) contains no \(\Lt\)-supported plane \(\pi\), but contains at least 2 disjoint lines, say \(L\) and \(M\) of \(\Lt\).
	\end{enumerate}
	\paragraph{Case 1}
	By Lemma \ref{lem:2supported}, the plane \(\pi\) contains \(q+1\) lines of \(\Lt\) which all meet in a point \(x\) and \(\pi\) contains no isolated points. Let \(y\) be the intersection of \(\pi\) and \(L\). Since all lines of \(\Lt\) through \(x\) are contained in \(\pi\) we know that \(x\neq y\). And because there are no isolated points we know that \(xy\in\Lt\).
	With two intersecting lines, the plane through \(xy\) and \(L\) must also be an \(\Lt\)-supported plane and contain \(q+1\) lines of \(\Lt\).
	If \(\Sigma\) would contain any more lines, they would also need to intersect \(xy\) since neither of the two \(\Lt\)-supported planes can have isolated points and we also cannot form a triangle.
	This would however give us three planes through \(xy\) of \(\Q\) in \(\Sigma\). This is only possible if \(\Sigma\subseteq \Q\) but that would contradict Lemma \ref{lem:SolidsInQ}. 
	
	The solid \(\Sigma\) is contained in the \(5\)-space \((xy)^\polarity\) and contains \(xy\). Hence, by Lemma \ref{lem:lineTangent}(\ref{lem:lineTangent:noIsoSub}), \(\Sigma\) contains no isolated points.
	
	\paragraph{Case 2}
	We consider three points \(x,y\) and \(z\) in the distance-\(3\)-trace \(T(L,M)\) of \(L\) and \(M\). If these three points were collinear in \(\PG\), then the line through them would be contained in \(\Q\) (Lemma \ref{lem:IntersectionLineQuadric}), which contradicts \(\delta(x,y)=\delta(y,z)=\delta(x,z) = 6\) (see Corollary \ref{cor:TraceOpposite} and Corollary \ref{cor:oppositeIffopposite}). Therefore, we know that these three points span a plane \(\pi\). 
	Since \(T(L,M) \subseteq L^\polarity \cap M^\polarity\) (see Lemma \ref{lem:pointTangent}(\ref{lem:pointTangent:points})), we have that \(\pi \subseteq L^\polarity \cap M^\polarity\). This means that \(L\) and \(M\) are contained in the \(4\)-dimensional space \(\pi^\polarity\). 
	Note that \(\pi^\polarity = x^\polarity\cap y^\polarity\cap z^\polarity\). 
	Hence, all lines of \(\T\) in this 4-dimensional subspace are in \( \Lt_{\le 3}(x)\cap \Lt_{\le 3}y) \cap \Lt_{\le 3}(z)\). Since there are only \(q+1\) lines in \(\Lt_{\le 3}(x)\cap \Lt_{\le 3}(y) = T(x,y)\), the subspace \(\pi^\polarity\) has at most \(q+1\) lines of \(\T\).
	
	We claim that at most one of the points \(x,y\) and \(z\) is contained in \(\pi\cap\pi^\polarity\). Assume to the contrary that both \(y\) and \(z\) were contained in \(\pi\cap\pi^\polarity\). The line \(yz\) would then be absolute since \(yz \subseteq \pi \implies \pi^\polarity \subseteq (yz)^\polarity\). This contradicts \(\delta(y,z) = 6\) and Corollary \ref{cor:oppositeIffopposite}.
	We may therefore assume without loss of generality that \(x\) and \(y\) are not in \(\pi^\theta\). 
	
	The point \(x\) is collinear with a point \(x_L\) on \(L\) and a point \(x_M\) on \(M\). We define \(\pi_x\) to be the plane spanned by \(x,x_L\) and \(x_M\). Similarly, we define \(\pi_y\) as the plane spanned by \(y, y_L = L\cap\Points_2(y)\) and \(y_M = M \cap \Points_2(y)\).
	
	Note that the lines \(x_Lx_M\) and \(y_Ly_M\) are precisely the intersections \(\pi_x \cap \pi^\polarity\) and \(\pi_y \cap \pi^\polarity\) respectively.
	
	Now let \(N\) be a line in \(R(L,M)\) different from \(L\) and \(M\). We define \(x_N\) and \(y_N\) to be the unique points on \(N\) collinear with respectively \(x\) and \(y\). Since the embedding is flat (Lemma \ref{lem:Flat}) in both \(x\) and \(y\), we know that \(x_N\in\pi_x\) and \(y_N\in\pi_y\).
	By the same argument as before, we also have that \(N\in \pi^\polarity\). This allows us to conclude that \(x_N\) is on the line \(x_Lx_M\) and similarly \(y_N\) is on the line \(y_Ly_M\).
	This shows that \(N\in \Sigma\). Since \(N\) was arbitrary, \(R(L,M)\subseteq \Sigma\).
	
	Finally, let \(v\) be an arbitrary point of \(\T\) in \(\Sigma\). For each point \(w_i\) (with \(i \in \{1,\dots,q^3+1\}\)) of \(T(L,M)\) we have that \(v \in \Pt_{\le 4}(w_i)\). 
	Since there are only \(q+1\) lines through \(v\), at least two of the shortest paths, say the path from \(v\) to \(w_j\) and the path from \(v\) to \(w_k\) (with \(j,k\in \{1,\dots, q^3+1\} \) and \(j\neq k\)), must use the same line though \(v\). 
	This shows that \(v\) is on \(R(L,M)\) and there are no isolated points.
\end{proof}

\begin{lemma}\label{lem:ThreeLinesIntersectingACommonLine}
	Let \(L_1,L_2 \) and \(L_3\) be three lines of \(\Lt\) intersecting a common line \(M\). If the intersection points of \(L_i, i= 1, 2, 3\) with \(M\) are different, then \(\langle L_1,L_2,L_3 \rangle\) is a \(4\)-dimensional subspace of \(\PG\) containing exactly \(q^2+q\) lines of \(\Lt\) intersecting \(M\). 
\end{lemma}
\begin{proof}
	We let \(e_i\) be the point with coordinates all 0 except \(x_i\) which we can choose to be \(1\).
	From Example \ref{ex:AbsoluteHexagon} we know that the points \(e_0, e_6, e_1, e_4, e_2\) and \( e_5\) form a hexagon of \(\T\). 
	Since \(\text{Aut}(\T)\) acts transitively on the set of ordered hexagons (see \cite[Chapter 4 \& 5]{VanMaldeghem1998}), 
	we may assume without loss of generality that \(e_0, e_6 \in M\), that \(e_0,e_5 \in L_1\) and that \(e_6, e_1 \in L_2\).
	
	We now consider a different line \(L_4\) of \(\Lt\) which intersects \(M\).
	
	Let \(y = (1,0,0,0,0,0,a,0)\) be the intersection of \(M\) and \(L_3\) and let \(y' = (1,0,0,0,0,0,b,0)\) be the intersection of \(M\) and \(L_4\).
	
	\begin{figure}[h]
		\begin{center}
			\begin{tikzpicture}[outer sep=3, inner sep=3]
				\foreach \P in {
					(0,0) coordinate (e0) node[anchor=east]{\((1,0,0,0,0,0,0,0) = e_0\)}, 
					(0,-1) coordinate (e6) node[anchor=east]{\((0,0,0,0,0,0,1,0) = e_6\)},  
					(0,-2) coordinate (y) node[anchor=east]{\((1,0,0,0,0,0,a,0) = y\)}, 
					(0,-3.2) coordinate (yp) node[anchor=east]{\((1,0,0,0,0,0,b,0) = y'\)},
					(2,1) coordinate (e5) node[anchor=west]{\(e_5 = (0,0,0,0,0,1,0,0)\)}, 
					(2,-0.5) coordinate (e1) node[anchor=west]{\(e_1 = (0,1,0,0,0,0,0,0)\)}, 
				}	
				\filldraw \P circle (3pt);
				
				\draw[thick] 
				(e0) -- (0,-4) node[anchor=west]{\(M\)}
				(e0) -- node[anchor=south]{\(L_1\)} (e5)
				(e6) -- node[anchor=south]{\(L_2\)} (e1)
				(y) -- node[anchor=south]{\(L_3\)} (2,-2)
				(yp) -- node[anchor=south]{\(L_4\)} (2,-3.5)
				;
				
			\end{tikzpicture}
		\end{center}	
		\caption{Coordinates}
	\end{figure}
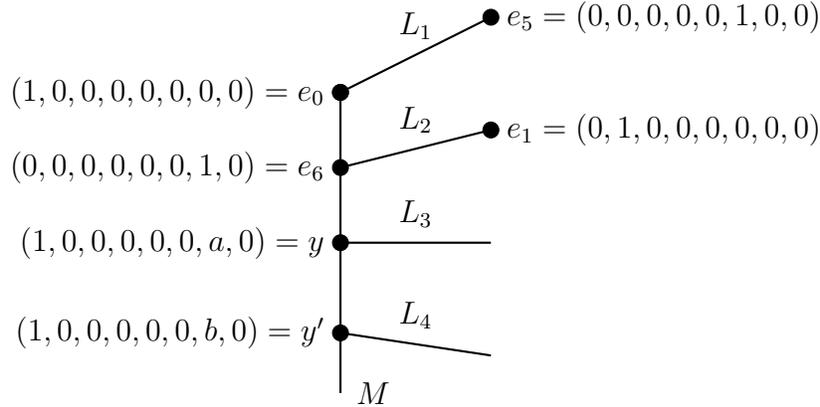
	
	We claim that \(L_4 \subseteq \langle L_1,L_2,L_3 \rangle\) if and only if there exists a \(\lambda\) in \(\F_q\) such that \(b = \lambda a\).
	
	As in Example \ref{ex:PlaneOfAbsoluteLines} we can derive the equations for the plane \(y^\triality\cap y^{\triality^2} = \langle M, L_3 \rangle \):
	
	\[\begin{cases}
		0 = x_2 \\
		0 = x_4 \\
		0 = x_1 - a^\sigma x_3 \\
		0 = x_1 + a^{\sigma^2} x_7 \\
		0 = x_7 + a^\sigma x_5 \\
		0 = x_3 - a^{\sigma^2} x_5. \\
	\end{cases}
	\]
	This plane \(\langle M, L_3\rangle\) contains the point \((0,a^\sigma a^{\sigma^2},0,a^{\sigma^2},0,1,0, -a^{\sigma}) =: z\). Similarly, we find the point \((0,b^\sigma b^{\sigma^2},0,b^{\sigma^2},0,1,0, -b^{\sigma}) =: z'\) in \(\langle M, L_4 \rangle\).
	
	The point \(z'\) is contained in \(\langle e_0, e_1, e_6, e_5, z\rangle\) if and only if there exists a \(\lambda \in \F_{q^3}\) such that:
	
	\[\begin{cases}
		\lambda a^\sigma = b^\sigma \\
		\lambda a^{\sigma^2} = b^{\sigma^2} \\
	\end{cases}
	\]
	
	or equivalently
	\[\begin{cases}
		\lambda a = b \\
		\lambda^\sigma = \lambda. \\
	\end{cases}
	\]
	
	Since \(\sigma\) fixes precisely \(\F_q\) in \(\F_{q^3}\), this proves that \(L_4 \subseteq \langle L_1,L_2,L_3 \rangle\) if and only if there exists a \(\lambda\) in \(\F_q\) such that \(b = \lambda a\). 
	
	In total this gives us exactly \(q+1\) planes through \(M\) each containing \(q\) lines of \(\Lt\) different from \(M\) which are contained in \(\langle L_1,L_2,L_3 \rangle\).
	
	Since solids contain at most \(2q+1\) lines of \(\Lt\) (Lemma \ref{lem:3supported}) the construction implies that \(\langle L_1,L_2,L_3 \rangle\) has dimension \(4\).
\end{proof}

Using the concept of ideal points of a subspace (Definition \ref{def:IsolatedAndIdeal}), we can derive the following simplification of Lemma \ref{lem:ThreeLinesIntersectingACommonLine}:
\begin{corollary}\label{cor:3thenqp1subpoints}
	Let \(U\) be a subspace of \(\PG\). Let \(L\) be a line of \(\Lt\) in \(U\).
	If \(L\) contains \(3\) ideal points of \(U\), then \(L\) contains at least \(q+1\) ideal points of \(U\).
\end{corollary}

\begin{lemma}\label{lem:4supported}
	\textbf{(4d)} Let \(U\) be a \(4\)-dimensional \(\Lt\)-supported subspace of \(\PG\). 
	
	Then, one of the following is true:
	\begin{itemize}
		\item \(|\Lt(U)| = q^2+q+1\).
		In this case there are \(q+1\) ideal points in \(U\). They are all collinear in \(\T\).	We can describe this situation as follows:
		there is a line \(L\in\Lt(U)\) and \(q+1\) planes through \(L\) in \(U\) each of which contain \(q\) lines of \(\Lt(U)\) different from \(L\).
		
		\item \(|\Lt(U)| = q^2+2q+1\).
		In this case there are \(q+2\) ideal points in \(U\). All but one of these are on a line of \(\Q\) which is not a line of \(\T\).
		We can describe this situation as follows:
		there are \(q+1\) planes in \(U\) each of which contain \(q+1\) lines of \(\Lt(U)\) and all these planes intersect in a point \(x\) of \(\Pt\). 
		Each line through \(x\) contains two ideal points, one of which is \(x\).
	\end{itemize}
	In either case, there are no isolated points.
\end{lemma}
\begin{proof}
	We write \(\Lt_U\) as an abbreviation of the set of lines \(\Lt(U)\).
	By Lemma \ref{lem:SupSubspace}, \(U\) contains either an \(\Lt\)-supported plane or an \(\Lt\)-supported solid. 
	
	Assume first that there is no \(\Lt\)-supported plane in \(U\). This would mean that none of the lines in \(\Lt_U\) intersect. We also know there exists an \(\Lt\)-supported solid in \(U\). By Lemma \ref{lem:3supported}, this solid contains \(q+1\) lines and no isolated points. Hence \(\Lt_U\) can have no further lines without having intersecting lines, which would contradict the assumption that \(U\) is \(\Lt\)-supported.
	
	We may therefore assume that \(\Lt_U\) contains an \(\Lt\)-supported plane \(\pi\). We define \(x\) to be the intersection point of all the lines of \(\Lt\) in \(\pi\).
	Since \(U\) is \(\Lt\)-supported, there exists a line \(L\) of \(\Lt_U\) not in \(\pi\).
	
	We investigate two different cases:
	\begin{enumerate}
		\item \(L\) does not intersect \(\pi\).
		\item \(L\) intersects \(\pi\).
	\end{enumerate}
	
	\paragraph{Case 1}
	
	Since all lines of \(\Lt\) through \(x\) are contained in \(\pi\), all lines at distance 3 from \(x\) contain a point of \(\pi\). Hence, because \(L\) does not intersect \(\pi\), we observe that \(\delta(x,L) = 5\).

	By Corollary \ref{cor:UniqueClosest}, \(L\) must contain a point \(y\) at distance \(4\) from \(x\). Let \(z:=x\bowtie y\) be the unique point collinear with \(x\) and \(y\).
	All points on \(L\) and in \(\Pt_{\le 2}(x)\) are contained in \(\Pt_{\le 4}(z)\) and therefore in \(z^\polarity\) (Lemma \ref{lem:pointTangent}(\ref{lem:pointTangent:points})). Since \(L\) and \(\pi\) span \(U\), this means that \(U\subseteq z^\polarity\). 
	Therefore, by Lemma \ref{lem:pointTangent}(\ref{lem:pointTangent:lines}), it follows that \(\Lt_U \subseteq \Lt_{\le 3}(z)\). 
	
	We now show that each line through \(z\) contains a second ideal point (different from \(z\)).
	Let \(M\) be a line through \(x\) different from \(xz\). This means that \(L\) and \(M\) are opposite in \(\T\). Let \(\Sigma\) be the solid spanned by \(L\) and \(M\). By Lemma \ref{lem:3supported}, the solid \(\Sigma\) contains the regulus \(R(L,M)\). Since \(z\in T(L,M)\), the point \(z\) is collinear in \(\T\) with a point on every line in \(R(L,M)\). With at least \(2\) lines of \(\Lt(U)\) through each of these points, they must be ideal in \(U\).
	Since the embedding is flat and \(z\not\in \Sigma \), we know that these points collinear with \(z\) and on \(R(L,M)\) are collinear (in \(\PG\)). 
	Let \(N\) be the line in $\PG$ through these \(q+1\) ideal points, which is clearly not a line of \(\T\). Since it contains at least \(q+1\) points of \(\Pt\) (and therefore of \(\Q\)), it must be a line of \(\Q\) (Lemma \ref{lem:IntersectionLineQuadric}).
	
	\begin{figure}[h]
		\begin{center}
			\resizebox{0.15\textwidth}{!}{
				\begin{tikzpicture}[outer sep=3, inner sep=3]
					\foreach \P in {
						(0,0) coordinate (p1) node[anchor=east]{\(z\)},
						(1,-2) coordinate (p2) node[anchor=north]{\(y\)}, 
						(1,2) coordinate (x) node[anchor=south]{\(x\)}, 
					}	
					\filldraw \P circle (3pt);
					
					\draw[thick]
					(p1) -- (x)
					(x) --node[anchor=south]{\(M\)} (3,2) 
					(x) --node[anchor=north]{} (3,1) 
					(p1) -- (p2)
					(p2) --node[anchor=south]{\(L\)} (3,-2);
					\draw[thick, dashed]
					(p2) --node[anchor=west]{\(N\)} (x);
				\end{tikzpicture}
			}
		\end{center}	
		\vspace{-1em}
		\caption{Case 1}
	\end{figure}
	
	By counting the lines through the ideal points in \(U\) on \(N\), we get \((q+1)^2=q^2+2q+1\) lines in \(\Lt_U\). Observe that this also counts all the lines through the ideal point \(z\) precisely once. 
	
	Since every line of \(\Lt\) in \(U\) intersects \(\Sigma\) and \(\Lt_U \subseteq \Lt_{\le 3}(z)\), we can use the fact there are no isolated points in \(\Sigma\) (see Lemma \ref{lem:3supported}) to conclude that there are no other lines in \(\Lt_U\).
	This also means there are no ideal points different from the \(q+2\) we already found.
	
	We know that \(U\) is a subspace of \(z^\polarity\) which contains \(\Pt_{\le 2}(z)\). By Lemma \ref{lem:pointTangent}(\ref{lem:pointTangent:noIsoSub}), \(U\) has no isolated points. 
	
	\paragraph{Case 2}
	If \(U\) contains any line \(L'\) which does not intersect \(\pi\), we can replace \(L\) by \(L'\) and end up in case 1. We may therefore assume all lines of \(\Lt\) in \(U\) intersect \(\pi\). 
	By Lemma \ref{lem:2supported}, the line \(L\) intersects one of the lines of \(\Lt\) through \(x\) in a point \(y\neq x\).
	This means that the points \(x\) and \(y\) are two ideal points on \(xy\). 
	The subspace \(U\) therefore must contain at least the \(2q+1\) lines through these two ideal points.
	Since \(U\) is \(\Lt\)-supported and all these \(2q+1\) lines are contained in a solid, there must still be another line in \(\Lt(U)\). That is, there exist a line \(M\) in \(\Lt(U)\) which goes through neither \(x\) nor \(y\).
	If \(M\) does not intersect \(xy\), we consider \(\langle L,xy\rangle \) and \(M\) to end up in case 1.
	If \(M\) does intersect \(xy\), then the line \(xy\) has at least three ideal points of \(U\).
	By Lemma \ref{lem:ThreeLinesIntersectingACommonLine}, the 4-space spanned by \(\pi\), \(L\) and \(M\), contains exactly \(q^2+q\) lines of \(\Lt\) intersecting \(xy\). Hence \(U\) contains precisely \(q^2+q+1\) lines of \(\Lt\).
	
	There are clearly \(q+1\) ideal points in \(U\) which are all on \(xy\).
	We also know that \(U\) is a subspace of \((xy)^\polarity\) which contains \(xy\). By Lemma \ref{lem:lineTangent}(\ref{lem:lineTangent:noIsoSub}), U has no isolated points. 
	
	\begin{figure}[h]
		\begin{center}
			\resizebox{0.15\textwidth}{!}{
				\begin{tikzpicture}[outer sep=3, inner sep=3]
					\foreach \P in {
						(0,0) coordinate (x) node[anchor=east]{\(x\)},
						(0,-1) coordinate (y) node[anchor=east]{\(y\)}, 
						(0,-2) coordinate (z)
					}	
					\filldraw \P circle (3pt);
					
					\draw[thick]
					(x) -- (0,-3)
					(x) --node[anchor=south]{\(N\)} (2,1) 
					(y) --node[anchor=north]{\(L\)} (2,-1)
					(z) --node[anchor=north]{\(M\)} (2,-2.5);
			\end{tikzpicture}}
		\end{center}	
		\vspace{-1em}
		\caption{Case 2}
	\end{figure}
	
\end{proof}

\begin{corollary}\label{cor:4dMax}
	\textbf{(4d')} Every \(4\)-dimensional subspace of \(\PG\) contains at most \(q^2+2q+1\) elements of \(\Lt\).
\end{corollary}

\begin{corollary}\label{cor:noHexIn4}
	Let \(U\) be a \(4\)-dimensional \(\Lt\)-supported subspace of \(\PG\), then \(U\) does not contain any hexagon of \(\T\).
\end{corollary}


We now turn our attention to \(\Lt\)-supported 5-spaces. A crucial part will be Lemma \ref{lem:hexToSub} which we prove in a few steps.

\begin{lemma}\label{lem:xL5}
	Let \(V\) be a 5-dimensional subspace of \(\PG\) containing an \(\Lt\)-supported plane \(\pi\) and  a line \(L\) of \(\Lt\) which does not intersect \(\pi\). 
	Then, there are at most two ideal points of \(V\) on \(L\).	
	\end{lemma}
	\begin{proof}
Assume there are three ideal points of \(V\) on \(L\). Lemma \ref{lem:ThreeLinesIntersectingACommonLine} implies that there must be a 4-dimensional subspace \(U\subseteq V\) through \(L\) which, by Lemma \ref{lem:lineTangent}(\ref{lem:lineTangent:lines}), is also a subspace of \(L^\polarity\). 
Let \(x\) be the unique ideal point of \(\pi\). 

The line \(L\) does not intersect \(\pi\) and therefore \(x\) is at distance \(5\) of \(L\) and not contained in \(U\subseteq L^\polarity\) (Lemma \ref{lem:lineTangent}(\ref{lem:lineTangent:points})). 
Since \(V\) has dimension 5, the plane \(\pi\) intersects \(U\) in a line. Since this line does not contain \(x\), we find \(q+1\) points in \(\Pt(U)\cap\Pt(\pi)\). 

However, since \(L\) does not intersect \(\pi\), we have \(\delta(L,x) = 5\) and hence there is a unique point in \(\Pt_{\le 3}(L)\cap \Pt_{\le 2}(x)\) (Corollary \ref{cor:UniqueClosest}). Lemma \ref{lem:lineTangent}(\ref{lem:lineTangent:points})
and Corollary \ref{cor:trialityBasic}(\ref{cor:trialityBasic:points}) imply that
\(\Pt_{\le 3}(L)\cap \Pt_{\le 2}(x) \supseteq \Pt(U)\cap\Pt(\pi)\). We have found a contradiction.

\end{proof}

\begin{corollary}\label{cor:Hex2Ideal}
Let \(V\) be a 5-dimensional subspace of \(\PG\) containing a hexagon \(\mathscr{H} = (L_1,L_2,L_3,L_4,L_5,L_6)\) of \(\T\). 
Then, for each of the lines \(L_i\) of \(\mathscr{H}\), there are exactly \(2\) ideal points of \(V\) on \(L_i\) (the points of the hexagon).
\end{corollary}

\begin{lemma}\label{lem:LandHex}
Let \(V\) be a 5-dimensional subspace of \(\PG\) containing a hexagon \(\mathscr{H} = (L_1,L_2,L_3,L_4,L_5,L_6)\) of \(\T\). Let \(M\) be a line of \(\Lt(V)\) which intersects one of the lines of \(\mathscr{H}\).

Then, there exists a hexagon \(\mathscr{H}'\) of \(\T\) through \(M\) in \(V\).
\end{lemma}
\begin{proof}
By Corollary \ref{cor:Hex2Ideal}, each line of \(\mathscr{H}\) has exactly two ideal points. Hence, without loss of generality, we may assume \(M\neq L_1\) and \(M\) intersects \(\mathscr{H}\) in the point \(L_1\cap L_2 =: x\). Let \(y\) be the intersection of \(L_4\) and \(L_5\). It is easy to see that \(\delta(M,y) = 5\) (Corollary \ref{cor:UniqueClosest}). 
Let \(x'\) be the unique point on \(M\) at distance \(4\) from \(y\) and let \(y' = x'\bowtie y\). Since \(y'\) is contained in \(\pi = \langle L_4,L_5 \rangle\) by property (Pl) (Lemma \ref{lem:pl}), we know that \(y'\) is contained in \(V\).
We have constructed the hexagon \(\mathscr{H}' = (M,x'y',y'y,L_5,L_6,L_1)\) and shown that this hexagon is contained in \(V\).
\end{proof}

\begin{lemma}\label{lem:hexToSub}
Let \(V\) be a 5-dimensional subspace of \(\PG\) containing a hexagon \(\mathscr{H} = (L_1,L_2,L_3,L_4,L_5,L_6)\) of \(\T\). 
Then, \(V\) intersects \(\T\) in a subhexagon of order \((1,q)\).
\end{lemma}
\begin{proof}
We first show that for every line \(L\) in \(\Lt(V)\), there are exactly \(2\) ideal points of \(V\) on \(L\).
We do this constructively and start with all lines through \(L_1\cap L_2 =: x\). 
Let \(M_{(1)}\) be such a line. By Lemma \ref{lem:LandHex}, \(M_{(1)}\) is contained in a hexagon in \(V\).
By Corollary \ref{cor:Hex2Ideal}, the line \(M_{(1)}\) has exactly \(2\) ideal points in \(V\). 

Next we go to all lines of \(\Lt_V\) at distance \(3\) from \(x\). Let \(M_{(3)}\) be such a line and let \(y\) be the unique point on \(M_{(3)}\) at distance 2 from \(x\). As a line through \(x\), the line \(xy\) has already been proven to be in a hexagon \(\mathscr{H}'\). Hence, we can again use Lemma \ref{lem:LandHex} and Corollary \ref{cor:Hex2Ideal}.

Finally, all lines at distance \(5\) from \(x\) again follow the same pattern since we have now also shown that all points at distance \(4\) from \(x\) are on a hexagon.

We now count the lines in \(\Lt(V)\). There are \(q+1\) lines through \(x\) which each contain a unique ideal point of \(V\) at distance \(2\) of \(x\). Hence there are \((q+1)q\) lines of \(\Lt_3(x)\) in \(V\) and also equally many points of \(\Pt_4(x)\) which are ideal in \(V\). By going one step further we get \((q+1)q^2\) lines of \(\Lt_5(x)\) in \(V\). These must be all lines of \(\Lt(V)\) and therefore:
\[
|\Lt(V)| = (q+1)q^2+ (q+1)q + (q+1) = (q+1)(q^2+q+1).
\]
Since \(\Lt\) contains no \(k\)-gons for \(k\in\{2,3,4,5\}\) we inherit this property in \(\Lt(V)\) and hence we can easily see that \(\Lt(V)\) (and the ideal points of \(V\)) has the structure of a hexagon of order \((1,q)\).
\end{proof}


\begin{lemma}\label{lem:supported5supported4}
Let \(V\) be a \(5\)-dimensional \(\Lt\)-supported subspace of \(\PG\) such that \(V\) contains two intersecting lines of \(\Lt\). Then, there exists a 4-dimensional subspace \(U\) of \(V\) which is \(\Lt\)-supported.
\end{lemma}
\begin{proof}
Let \(\pi\) be a plane through two intersecting lines of \(\Lt(V)\). Since \(V\) is \(\Lt\)-supported, there must exist a line \(L\) of \(\Lt(V)\) which is not contained in \(\pi\). If \(L\) does not intersect \(\pi\), then the \(\Lt\)-supported subspace \(\langle \pi ,L\rangle\) has dimension 4 and we are done. If \(L\) does intersect \(\pi\), then we can find another line \(M\) of \(\Lt(V)\) not in \(\langle \pi ,L \rangle\) where we again use that \(V\) is \(\Lt\)-supported. Either \(\langle \pi ,M \rangle\) or \(\langle \pi ,L ,M \rangle\) will now be a 4-dimensional \(\Lt\)-supported subspace.
\end{proof}

\begin{lemma}\label{lem:LM5}
Let \(V\) be a 5-dimensional subspace of \(\PG\) containing two lines \(L\) and \(M\) of \(\Lt\) with \(\delta(L,M) = 4\). Then, there are at most \(q+1\) ideal points of \(V\) on \(L\). 
\end{lemma}
\begin{proof}
Assume that there are more than \(q+1\) ideal points of \(V\) on \(L\). 
This means we find more than \(q^2+q\) lines in \(V\), different from \(L\), intersecting \(L\).
By Lemma \ref{lem:ThreeLinesIntersectingACommonLine}, the subset of all lines in \(\Lt(V)\) which intersect \(L\) must span the full 5-dimensional space \(V\). Since they are also all contained in \(L^\polarity\), it follows that \(V = L^\polarity\) but this contradicts the existence of \(M\) in \(\Lt(V)\) (see Lemma \ref{lem:lineTangent}(\ref{lem:lineTangent:lines})).
\end{proof}

A more complete description will be proven later in Lemma \ref{lem:5supported}, but for now we can state the following:
\begin{lemma}\label{lem:5supportedInt}
Let \(V\) be a \(5\)-dimensional \(\Lt\)-supported subspace of \(\PG\) containing two intersecting lines of \(\Lt\). 

Then, one of the following is true:
\begin{itemize}
	
	\item \(|\Lt(V)| = q^3+q^2+q+1\).
	
	More specifically, there is a point \(x\) in \(V\) such that all lines of \(\Lt(V)\) are contained in \(\Lt_{\le 3}(x)\) and each line through \(x\) intersects \(q^2+q\) of the other lines.
	All ideal points of \(V\) are on lines of \(\Lt(V)\) though \(x\) and therefore coplanar.
	Each of the lines through \(x\) contains precisely \(q+1\) ideal points of \(V\).
	
	\item \(|\Lt(V)| = q^3+2q^2+2q+1\).
	
	More specifically, all lines of \(\Lt(V)\) (together with the ideal points of \(V\)), form a subhexagon of order \((1,q)\).
	
	\item \(|\Lt(V)| = q^4+q+1\).
	
	More specifically, there exists a line \(M\) in \(\Lt(V)\) such that \(\Lt(V) = \Lt_{\le 2}(M)\).
	There are \(q^3+1\) ideal points in \(V\), they are precisely all the points on \(M\).
\end{itemize}
\end{lemma}
\begin{proof}
We write \(\Lt_V\) as an abbreviation of the set of lines \(\Lt(V)\).

By Lemma \ref{lem:supported5supported4} and the fact that \(V\) is \(\Lt\)-supported, we may assume that there exist an \(\Lt\)-supported 4-dimensional space \(U\) and a line \(L\) of \(\Lt_V\) not in \(U\) such that \(V = \langle U, L\rangle\).

Given Lemma \ref{lem:4supported}, we can consider the following cases:
\begin{enumerate} 
	\item \(|\Lt(U)| = q^2+q+1\). We let \(M\) be the unique line in \(\Lt(U)\) which intersects all other lines of \(\Lt(U)\).
	\begin{enumerate} 
		\item \(L\) intersects \(M\).
		\item \(L\) does not intersect \(M\).
	\end{enumerate}
	\item \(|\Lt(U)| = q^2+2q+1\). We let \(x\) be the unique point in all the \(\Lt\)-supported planes of \(U\) and \(\pi\) the unique plane through \(x\) in \(U\) such that all lines of \(\Lt(U)\) intersect \(\pi\).
	\begin{enumerate} 
		\item \(L\) intersects \(\pi\).
		\item \(L\) does not intersect \(\pi\).
	\end{enumerate}
\end{enumerate}

\paragraph{Case 1.a: \(|\Lt(U)| = q^2+q+1\) and \(L\) intersects \(M\)}
In this case, both \(U\) and \(L\) are contained in the 5-dimensional subspace \(M^\polarity\) and therefore \(V = M^\polarity\). By Lemma \ref{lem:lineTangent}(\ref{lem:lineTangent:lines}), this subspace contains precisely the lines \(\Lt_{\le 2}(M)\). This gives us \(q^4+q+1\) lines and no isolated points.  
Each point on \(M\) is an ideal point of \(V\) and there are no other ideal points.

\paragraph{Case 1.b: \(|\Lt(U)| = q^2+q+1\) and \(L\) does not intersect \(M\)}
By Lemma \ref{lem:4supported}, \(U\) contains no isolated points. Therefore, we can let \(M_{x,1}\) be the unique line intersecting \(L\) and \(M\) where \(x\) refers to the intersection of \(M_{x,1}\) and \(M\). We already know that all lines of \(\Lt\) through \(x\) are contained in \(\Lt_V\) (property (Pl)).
We enumerate these lines through \(x\) as \(M_{x,i}\) with \(i\in\{1,2,\dots,q+1\}\) and \(M_{x,q+1} = M\).
We now want to show that there are exactly \(q+1\) ideal points of \(V\) on each of the lines \(M_{x,i}\). 

We already know that there are \(q+1\) ideal points on \(M\) in \(U\) and therefore there are also at least \(q+1\) ideal points on \(M\) in \(V\).
We let \(y\) and \(z\) be two ideal points of \(V\) on \(M\) different from \(x\). 
We label the lines through \(y\) and \(z\) as \(M_{y,i}\) and \(M_{z,i}\) respectively with again \(i\in \{1,2,\dots,q+1\}\) and \(M_{y,q+1} = M = M_{z,q+1}\).  

If we consider any line \(M_{x,i}\) with \(i\in \{2,3\dots,q+1\}\) together with \(L\), then Lemma \ref{lem:LM5} implies that there are at most \(q+1\) ideal points on \(M_{x,i}\). For \(M_{x,1}\) we cannot use \(L\) since \(\delta(M_{x,1}, L)\) is \(2\) instead of \(4\), but we can use \(M_{y,1}\) to obtain the same result from Lemma \ref{lem:LM5}.
Hence, every line of \(\Lt\) through \(x\) contains at most \(q+1\) ideal points in \(V\).

Now let \(M_{x,i}\) be a line of \(\Lt\) through \(x\) different from \(M\) and \(M_{x,1}\). 

Each line \(M_{y,j}\) with \(j\in \{1,\dots,q\}\) together with the line \(L\) defines a regulus \(R(M_{y,j},L)\) contained in \(\Lt_V\) (see Lemma \ref{lem:3supported}).
Since \(x\in T(M_{y,j},L)\) this regulus also contains a line \(N_{y,j}\) intersecting \(M_{x,i}\). Similarly, each line \(M_{z,k}\) with \(k\in \{1,\dots,q\}\) together with the line \(L\) gives us a line \(N_{z,k}\) intersecting \(M_{x,i}\) in a point different from \(x\).
Since two reguli can have at most one line in common (see Remark \ref{rem:d3reg}), we know that \(N_{y,j}\in R(M_{y,j},L)\) and \(N_{z,k} \in R(M_{z,k},L)\) have to be different for all \(j\in \{1,\dots,q\}\) and \(k\in \{1,\dots,q\}\).

In total we have found \(2q\) lines intersecting \(M_{x,i}\) not in \(x\). Together with all the lines through \(x\), it follows from Lemma \ref{lem:3supported} and Lemma \ref{lem:4supported} that there are at least \(q+1\) ideal points of \(V\) on \(M_{x,i}\).
Finally, we can swap out \(L\) and \(N_{y,1}\) to find at least \(q+1\) ideal points on \(M_{x,1}\).

Observe that we have constructed an \(\Lt\)-supported 4-dimensional subspace through each of the lines \(M_{x,i}\). Since none of these has any isolated points (Lemma \ref{lem:4supported}), we can see that \(\Lt_V\) cannot contain any lines that do not meet a line through \(x\). 
By counting all the lines we find \(|\Lt_V| = q^3+q^2+q+1\). 

\paragraph{Case 2.a: \(|\Lt(U)| = q^2+2q+1\) and \(L\) intersects \(\pi\)}
Let \(y\) be the intersection point of \(L\) and \(\pi\).
Since there are now three lines in \(\Lt_V\) intersecting \(xy\) in different points, we use Lemma \ref{lem:ThreeLinesIntersectingACommonLine} to find an \(\Lt\)-supported 4-dimensional space in \(V\) containing \(q^2+q+1\) lines. This brings us back to case 1 (and in particular we will end up in case 1.b).

\paragraph{Case 2.b: \(|\Lt(U)| = q^2+2q+1\) and \(L\) does not intersect \(\pi\)}
There exist a unique line through \(x\) which is at distance 4 from \(L\) (see Corollary \ref{cor:UniqueClosest}).
Observe that the unique line which intersects \(\pi\) and \(L\) is contained in \(V\) since both these intersections are in \(V\).
Now let \(M\) be a line though \(x\) opposite \(L\) and let \(y\) be the second ideal point on \(M\) in \(U\) (besides \(x\)).
By Corollary \ref{cor:UniqueClosest}, there is a unique point \(z\) on \(L\) at distance 4 from \(y\). Since \(y\) is ideal in \(V\), the point \(y\bowtie z =: c\) collinear with \(y\) and \(z\) is contained in \(V\). Hence also the line \(cz\) is in \(V\).
This proves the existence of a hexagon though \(M\) and \(L\) in \(V\).

By Lemma \ref{lem:hexToSub} this implies that \(V\) intersects \(\T\) in a subhexagon of order \((1,q)\).
This gives us \(q^3+2q^2+2q^2+1\) lines in \(\Lt_V\) and no isolated points.
\end{proof}

\begin{corollary}\label{cor:pathIn5IsSub}
Let \(V\) be a \(5\)-dimensional \(\Lt\)-supported subspace of \(\PG\) which contains a path of five lines (where we assume all five lines and the successive intersection points to be distinct). Then, \(V\) intersects \(\T\) in a subhexagon of order \((1,q)\).
\end{corollary}


Before we describe \(5\)-dimensional \(\Lt\)-supported subspaces such that all lines of \(\Lt\) are disjoint, we first turn our attention to 6-dimensional \(\Lt\)-supported subspaces. Similar to Lemma \ref{lem:hexToSub}, a crucial part will be Lemma \ref{lem:hepToSC}. To prove this result, we take similar steps as in Lemma \ref{lem:xL5} and Lemma \ref{lem:LandHex}.

\begin{lemma}\label{lem:xL6}
Let \(W\) be a 6-dimensional subspace of \(\PG\) containing an \(\Lt\)-supported plane \(\pi\) and  a line \(L\) of \(\Lt\) which does not intersect \(\pi\). Then, there are at most \(q+1\) ideal points of \(W\) on \(L\). 
\end{lemma}
\begin{proof}

Assume there are more than \(q+1\) ideal points of \(W\) on \(L\). Lemma \ref{lem:4supported} shows that there can be at most \(q+1\) ideal points on a line in a 4 dimensional subspace. Hence, there must be a 5-dimensional subspace \(V\subseteq W\) through \(L\) which, by Lemma \ref{lem:lineTangent}(\ref{lem:lineTangent:lines}), is also a subspace of \(L^\polarity\). 
Let \(x\) be the unique ideal point of \(\pi\).

The line \(L\) does not intersect \(\pi\) and therefore \(x\) is at distance \(5\) of \(L\) and not contained in \(V\subseteq L^\polarity\) (Lemma \ref{lem:lineTangent}(\ref{lem:lineTangent:points})). Since \(W\) has dimension 6, the plane \(\pi\) intersects \(V\) in a line. Since this line does not contain \(x\), we find \(q+1\) points in \(\Pt(V)\cap\Pt(\pi)\). 

However, since \(L\) does not intersect \(\pi\), we have \(\delta(L,x) = 5\) and hence there is a unique point in \(\Pt_{\le 3}(L)\cap \Pt_{\le 2}(x)\) (Corollary \ref{cor:UniqueClosest}). Lemma \ref{lem:lineTangent}(\ref{lem:lineTangent:points})
and Corollary \ref{cor:trialityBasic}(\ref{cor:trialityBasic:points}) imply that
\(\Pt_{\le 3}(L) \cap \Pt_{\le 2}(x) \supseteq \Pt(V)\cap\Pt(\pi)\). We have found a contradiction.
\end{proof}

\begin{corollary}\label{cor:Hepq+1}
Let \(W\) be a 6-dimensional subspace of \(\PG\) containing a heptagon \(\mathscr{H} = \{L_1,L_2,L_3,L_4,L_5,L_6,L_7\}\) of \(\T\). 
Then, for each of lines \(L_i\) of \(\mathscr{H}\) there are exactly \(q+1\) ideal points of \(W\) on \(L_i\).
\end{corollary}
\begin{proof}
Because of symmetry it suffices to show the required property for \(L_1\). Let \(x\) be the intersection point of \(L_4\) and \(L_5\).
By Corollary \ref{cor:UniqueClosest}, there is a unique point \(y\) on \(L_1\) at distance 4 from \(x\). Clearly \(y\) is not also on \(L_2\) or \(L_7\) since \(L_2\cap L_1\) and \(L_7\cap L_1\) are both at distance \(6\) from \(x\). Let \(z = x \bowtie y\) be the unique point collinear with both \(x\) and \(y\).  Property (Pl) (Lemma \ref{lem:pl}) shows that \(xz\) is contained in \(W\). It follows that \(z\) and therefore also \(yz\) is contained in \(W\). Hence, we have found \(3\) different ideal points on \(L_1\). By Corollary \ref{cor:3thenqp1subpoints} there must be at least \(q+1\) ideal points on \(L_1\).
Lemma \ref{lem:xL6} completes the proof.
\end{proof}

\begin{lemma}\label{lem:LandHep}
Let \(W\) be a 6-dimensional subspace of \(\PG\) containing a heptagon \(\mathscr{H} = \{L_1,L_2,L_3,L_4,L_5,L_6,L_7\}\) of \(\T\). Let \(M\) be a line in \(\Lt(V)\) intersecting one of the lines of \(\mathscr{H}\). 

Then, there exists a heptagon \(\mathscr{H}'\) of \(\T\) through \(M\) in \(V\).
\end{lemma}
\begin{proof}
We may assume that \(M\) intersects \(L_1\) and \(M\not\in \{L_1,L_2,L_7\}\).
Let \(x = L_1 \cap M\).
Symmetry allows us to assume that \(x\) is not on \(L_7\).
By Corollary \ref{cor:UniqueClosest}, there is at least one point \(y\) on \(M\) at distance 4 from the intersection point \(z\) of \(L_4\) and \(L_5\). Let \(c\) be the unique point collinear with \(y\) and \(z\).
We split into cases based on \(x\) and \(c\):
\begin{enumerate}
	\item \(x \not\in L_2\)
	\begin{enumerate}
		\item  \(c\not \in M\)
		\item  \(c \in M\)
	\end{enumerate}
	\item \(x \in L_2\)
	\begin{enumerate}
		\item  \(c\not \in L_5\)
		\item  \(c \in L_5\)
	\end{enumerate}
\end{enumerate}

\paragraph{Case 1.a: \(x \not\in L_2\), \(c\not \in M\)}
Since \(c \neq z\) but collinear with \(z\), without loss of generality we may assume that \(c\not\in L_5\). We have constructed the heptagon \[\{M,yc,cz,L_5,L_6,L_7,L_1\}.\]
Property (Pl) (Lemma \ref{lem:pl}) guarantees that \(c\) and therefore this heptagon is contained in \(V\). Note that it is possible that \(cz = L_4\).

\begin{figure}[h]
	\begin{center}
		\resizebox{0.25\textwidth}{!}{
			\begin{tikzpicture}[outer sep=3, inner sep=3]
				\foreach \P in {
					(0:2.5) coordinate (p1),
					(1*360/7:2.5) coordinate (p2), 
					(2*360/7:2.5) coordinate (p3), 
					(3*360/7:2.5) coordinate (p4), 
					(4*360/7:2.5) coordinate (p5) node[anchor=east]{\(z\)} , 
					(5*360/7:2.5) coordinate (p6), 
					(6*360/7:2.5) coordinate (p7), 
				}	
				\filldraw \P circle (3pt);
				\coordinate (x) at ($(p1)!0.5!(p2)$);
				\filldraw (x) circle (3pt) node[anchor=south east ]{\(x\)};
				\filldraw (5.5*360/7:1) coordinate (c) circle (3pt) node[anchor=west]{\(c\)} ;
				\filldraw (0,0) coordinate (y) circle (3pt) node[anchor=south east]{\(y\)} ;
				
				\draw[thick]
				(p1) -- node[anchor= west]{\(L_1\)} 
				(p2) --node[anchor=south]{\(L_2\)} 
				(p3) --node[anchor=south east]{\(L_3\)} 
				(p4) --node[anchor=east]{\(L_4\)} 
				(p5) --node[anchor=north east]{\(L_5\)} 
				(p6) --node[anchor=north]{\(L_6\)} 
				(p7) --node[anchor=west]{\(L_7\)} 
				(p1)
				(x) --node[anchor=south]{\(M\)}  (0,0)
				(0,0) -- (c)
				(c) -- (p5);
				;
		\end{tikzpicture}}
	\end{center}	
	\vspace{-1em}
	\caption{Case 1.a}
\end{figure}
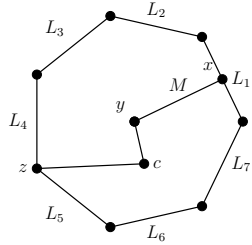

\paragraph{Case 1.b: \(x \not\in L_2\), \(c \in M\)}
Note that \(c\) cannot be on \(L_5\) since otherwise we would find a pentagon through the lines \(L_1, L_7, L_6, L_5\) and \(M\).
Let \(x'\) be the intersection point of \(L_5\) and \(L_6\).
Let \(z'\) be the unique point on \(L_2\) at distance 4 from \(x'\) (Corollary \ref{cor:UniqueClosest}) and let \(c' = x'\bowtie z'\) be the unique point collinear with \(x'\) and \(z'\). We can see that \(z'\neq L_1 \cap L_2\) since otherwise we would find a pentagon. Similarly, \(c'\) is not on \(L_5\).

We have constructed the heptagon \[\{M,cz,L_5,x'c',c'z',L_2,L_1\}.\]
All lines through \(x'\) are contained in \(V\) and therefore also \(x'c'\). Similarly, \(cz\) is a line in \(V\). It follows that this heptagon is contained in \(V\).

\begin{figure}[h]
	\begin{center}
		\resizebox{0.25\textwidth}{!}{
			\begin{tikzpicture}[outer sep=3, inner sep=3]
				\foreach \P in {
					(0:2.5) coordinate (p1),
					(1*360/7:2.5) coordinate (p2), 
					(2*360/7:2.5) coordinate (p3), 
					(3*360/7:2.5) coordinate (p4), 
					(4*360/7:2.5) coordinate (p5) node[anchor=east]{\(z\)}, 
					(5*360/7:2.5) coordinate (p6) node[anchor=north]{\(x'\)}, 
					(6*360/7:2.5) coordinate (p7) 
				}	
				\filldraw \P circle (3pt);
				
				\coordinate (x) at ($(p1)!0.5!(p2)$);
				\filldraw (x) circle (3pt) node[anchor=south east]{\(x\)};
				\filldraw (5.5*360/7:1) coordinate (c) circle (3pt) node[anchor=west]{\(c\)} ;
				\filldraw (-0.5,0.5) coordinate (cp) circle (3pt) node[anchor=south east]{\(c'\)} ;
				\coordinate (xp) at ($(p2)!0.5!(p3)$);
				\filldraw (xp) circle (3pt) node[anchor=north]{\(z'\)};
				
				\draw[thick]
				(p1) --node[anchor=west]{\(L_1\)} 
				(p2) --node[anchor=south]{\(L_2\)} 
				(p3) --node[anchor=south east]{\(L_3\)} 
				(p4) --node[anchor=east]{\(L_4\)} 
				(p5) --node[anchor=north east]{\(L_5\)} 
				(p6) --node[anchor=north]{\(L_6\)} 
				(p7) --node[anchor=west]{\(L_7\)} 
				(p1)
				(x) --node[anchor=south]{\(M\)}  (c)
				(c) -- (p5)
				(p6) -- (cp)
				(cp) -- (xp)
				;
			\end{tikzpicture}
		}
	\end{center}	
	\vspace{-1em}
	\caption{Case 1.b}
\end{figure}
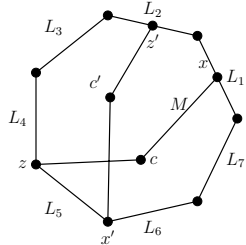

\paragraph{Case 2.a: \(x \in L_2\), \(c\not \in L_5\)}
We have constructed the heptagon \[\{M,yc,cz,L_5,L_6,L_7,L_1\}.\]
Property (Pl) (Lemma \ref{lem:pl}) guarantees that \(c\) and \(y\) are contained in \(V\). Hence, this heptagon is contained in \(V\).
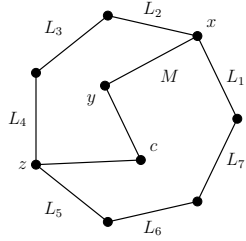
\begin{figure}[h]
	\begin{center}
		\resizebox{0.25\textwidth}{!}{
			\begin{tikzpicture}[outer sep=3, inner sep=3]
				\foreach \P in {
					(0:2.5) coordinate (p1),
					(1*360/7:2.5) coordinate (p2) node[anchor=south west]{\(x\)}, 
					(2*360/7:2.5) coordinate (p3), 
					(3*360/7:2.5) coordinate (p4), 
					(4*360/7:2.5) coordinate (p5) node[anchor=east]{\(z\)}, 
					(5*360/7:2.5) coordinate (p6), 
					(6*360/7:2.5) coordinate (p7) 
				}	
				\filldraw \P circle (3pt);
				
				\coordinate (c) at (5.5*360/7:1);
				\filldraw (c) circle (3pt) node[anchor=south west]{\(c\)};
				\coordinate (y) at (2.5*360/7:1);
				\filldraw (y) circle (3pt) node[anchor=north east]{\(y\)};
				
				\draw[thick]
				(p1) --node[anchor=west]{\(L_1\)} 
				(p2) --node[anchor=south]{\(L_2\)} 
				(p3) --node[anchor=south east]{\(L_3\)} 
				(p4) --node[anchor=east]{\(L_4\)} 
				(p5) --node[anchor=north east]{\(L_5\)} 
				(p6) --node[anchor=north]{\(L_6\)} 
				(p7) --node[anchor=west]{\(L_7\)} 
				(p1)
				(p5) -- (c)
				(c) -- (y)
				(y) --node[anchor=north west]{\(M\)}  (p2)
				;
			\end{tikzpicture}
		}
	\end{center}	
	\caption{Case 2.a}
\end{figure}

\paragraph{Case 2.b: \(x \in L_2\), \(c \in L_5\)}
By Corollary \ref{cor:Hepq+1} there exists an ideal point \(x'\) on \(L_6\) not on \(L_5\) and not on \(L_7\). Let \(z'\) be the unique point on \(L_2\) at distance 4 from \(x'\) (Corollary \ref{cor:UniqueClosest}) and let \(y' = x'\bowtie z'\) be the unique point collinear with \(x'\) and \(z'\). Note that it is possible that \(z'\) is on \(L_3\).
We have constructed the heptagon \[\{M,yc,L_5,L_6,x'y',y'z',L_2\}.\]

All lines through \(x'\) are contained in \(V\) and therefore also \(x'y'\), showing that this heptagon is contained in \(V\).

\begin{figure}[h]
	\begin{center}
		\resizebox{0.25\textwidth}{!}{
			\begin{tikzpicture}[outer sep=3, inner sep=3]
				\foreach \P in {
					(0:2.5) coordinate (p1),
					(1*360/7:2.5) coordinate (p2) node[anchor=west]{\(x\)}, 
					(2*360/7:2.5) coordinate (p3), 
					(3*360/7:2.5) coordinate (p4), 
					(4*360/7:2.5) coordinate (p5) node[anchor=east]{\(z\)}, 
					(5*360/7:2.5) coordinate (p6), 
					(6*360/7:2.5) coordinate (p7) 
				}	
				\filldraw \P circle (3pt);
				\coordinate (c) at ($(p5)!0.5!(p6)$);
				\filldraw (c) circle (3pt) node[anchor=south west]{\(c\)};
				\coordinate (y) at (3*360/7:1);
				\filldraw (y) circle (3pt) node[anchor=north east]{\(y\)};
				\coordinate (xp) at ($(p2)!0.5!(p3)$);
				\filldraw (xp) circle (3pt) node[anchor=north east]{\(z'\)};
				\coordinate (yp) at (6.5*360/7:1.5);
				\filldraw (yp) circle (3pt) node[anchor= west]{\(y'\)};
				\coordinate (cp) at ($(p6)!0.5!(p7)$);
				\filldraw (cp) circle (3pt) node[anchor= south]{\(x'\)};
				
				\draw[thick]
				(p1) --node[anchor=west]{\(L_1\)} 
				(p2) --node[anchor=south]{\(L_2\)} 
				(p3) --node[anchor=south east]{\(L_3\)} 
				(p4) --node[anchor=east]{\(L_4\)} 
				(p5) --node[anchor=north east]{\(L_5\)} 
				(p6) --node[anchor=north]{\(L_6\)} 
				(p7) --node[anchor=west]{\(L_7\)} 
				(p1)
				(c) -- (y)
				(y) --node[anchor=north ]{\(M\)}  (p2)
				(cp) -- (yp) -- (xp)
				;
			\end{tikzpicture}
		}
	\end{center}
	\caption{Case 2.b}	
\end{figure}
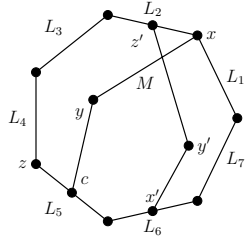

\end{proof}

The following proof is similar to the proof of Lemma \ref{lem:hexToSub} but since it deviates at a few important points, we include it here.

\begin{lemma}\label{lem:hepToSC}
Let \(W\) be a 6-dimensional subspace of \(\PG\) containing a heptagon \(\mathscr{H} = (L_1,L_2,L_3,L_4,L_5,L_6,L_7)\) of \(\T\). 
Then, \(W\) intersects \(\T\) in a split Cayley hexagon of order \((q,q)\). This split Cayley hexagon is regularly embedded in a subspace \(\mathsf{PG}(6,q)\) of \(W\). 
\end{lemma}
\begin{proof}
We follow similar steps as in Lemma \ref{lem:hexToSub}.

We first show that for every line \(L\) in \(\Lt(W)\), there are exactly \(q+1\) ideal points of \(W\) on \(L\).
We do this constructively and start with all lines through \(L_1\cap L_2 =: x\). 
Let \(M_{(1)}\) be such a line. By Lemma \ref{lem:LandHep}, \(M_{(1)}\) is contained in a heptagon in \(W\).
By Corollary \ref{cor:Hepq+1}, the line \(M_{(1)}\) has exactly \(q+1\) ideal points in \(W\). 

Next we go to all lines of \(\Lt_W\) at distance \(3\) from \(x\). Let \(M_{(3)}\) be such a line and let \(y\) be the unique point on \(M_{(3)}\) at distance 2 from \(x\). As a line through \(x\), the line \(xy\) has already been proven to be in a heptagon \(\mathscr{H}'\). Hence, we can again use Lemma \ref{lem:LandHep} and Corollary \ref{cor:Hepq+1}.

Finally, all lines at distance \(5\) from \(x\) in \(W\) again follow the same pattern since we have now also shown that all ideal points at distance \(4\) from \(x\) are on a heptagon.

We now count the lines in \(\Lt(W)\). There are \(q+1\) lines through \(x\) which each contain \(q\) ideal point of \(W\) at distance \(2\) of \(x\). Hence there are \((q+1)q^2\) lines of \(\Lt_3(x)\) in \(W\) and \((q+1)q^3\) points of \(\Pt_4(x)\) which are ideal in \(W\). By going one step further we get \((q+1)q^4\) lines of \(\Lt_5(x)\) in \(W\). By Lemma \ref{lem:pointTangent}(\ref{lem:pointTangent:lines}), these must be all lines of \(\Lt(W)\) and therefore:
\[
|\Lt(V)| = (q+1)q^4 + (q+1)q^2 + (q+1) = (q+1)(q^4+q^2+1).
\]
Since \(\Lt\) contains no \(k\)-gons for \(k\in\{2,3,4,5\}\) we inherit this property in \(\Lt(W)\) and hence we can easily see that \(\Lt(W)\) (and the ideal points of \(W\)) has the structure of a generalised hexagon \(\Gamma\) of order \((q,q)\).

Since the embedding of \(\T\) is flat in \(\PG\) (see Lemma \ref{lem:Flat}), the embedding of \(\Gamma\) is clearly flat in \(W\). The embedding is also lax (the points of the embedded generalised hexagon on a line are collinear in \(W\)).
We now show that embedding is polarised (for any point \(z\) of \(\Gamma\) the set of points not opposite \(z\) in \(\Gamma\) does not span \(W\)).
By the definition \(\Gamma\), the point \(z\) is also a point of \(\T\) and by Lemma \ref{lem:pointTangent}(\ref{lem:pointTangent:points}), the set of all points not opposite \(z\) in \(\T\) are in the hyperplane \(z^\polarity\) of \(\PG\). Since \(W\) contains a heptagon and \(z^\polarity\) does not, \(W\neq z^\polarity\) and \(W\cap z^\polarity\) has dimension at most 5. Hence all the points of \(\Gamma\) which are not opposite \(z\), which are contained in \(W\cap z^\polarity\), cannot span \(W\).
The embedding is clearly polarised.

Theorem \ref{thm:KeyTheorem}(1) implies that \(\Gamma\) is naturally embedded in a subspace \(\mathsf{PG}(6,q)\) of \(W\) and by Theorem \ref{thm:KeyTheorem}(2), \(\Gamma\) must be a split Cayley hexagon.
\end{proof}

We will later see that the requirement for the following statement is always satisfied, see Corollary \ref{cor:6supported}.
\begin{lemma}\label{lem:6supportedInt} 
Let \(W\) be a \(6\)-dimensional \(\Lt\)-supported subspace of \(\PG\) such that

\(W\) contains two intersecting lines of \(\Lt\).

Then, one of the following is true:
\begin{itemize}
	\item \(|\Lt(W)| = q^5+q^4+q+1\).
	More specifically, there is a point \(x\) in \(W\) such that \(\Lt(W) = \Lt_{\le 3}(x)\).
	The ideal points of \(W\) are precisely the points \(\Pt_{\le 2}(x)\).
	\item \(|\Lt(W)| = q^5+q^4+q^3+q^2+q+1\).
	More specifically, all lines of \(\Lt(W)\) (and the ideal points of \(W\)) form a split Cayley subhexagon. This split Cayley hexagon is regularly embedded in a subspace \(\mathsf{PG}(6,q)\) of \(W\). 
\end{itemize}
In either case, \(W\) does not contain any isolated points.
\end{lemma}

\begin{proof}
We write \(\Lt_W\) as an abbreviation of the set of lines \(\Lt(W)\). 
Let \(x\) be a point of \(W\). If all points of \(W\) on lines of \(\Lt_W\) are contained in \(x^\polarity\), then \(W = x^\polarity\) and we find \(q^5+q^4+q+1\) lines of \(\Lt\) in \(W\) and no isolated points (see Lemma \ref{lem:pointTangent}).
Hence, we may assume that for any point \(x\) in \(W\) there exists another point \(x'\) in \(W\) on a line of \(\Lt_W\) such that \(x\) and \(x'\) are opposite. Consequently for every point in \(W\) there is a line of \(\Lt(W)\) at distance 5 from it. 

We want to use this observation to construct a heptagon in \(W\) and apply Lemma \ref{lem:hepToSC}. To accomplish this goal, we first construct a hexagon as an intermediary step.

The assumption of this lemma implies that \(\Lt_W\) contains two lines intersecting in a point \(x_1\). Since the embedding is flat, this is equivalent with \(x_1\) being an ideal point of \(W\). Let \(L\) be a line of \(\Lt_W\) at distance 5 from \(x_1\). All lines of \(\Lt\) through \(x_1\) are contained in \(\Lt_W\). In particular, the unique point \(x_2\) collinear with \(x_1\) and at distance 3 from \(L\) (Corollary \ref{cor:UniqueClosest}) is contained in \(W\). The unique line through \(x_2\) intersecting \(L\) contains two points of \(W\) and is therefore also contained in \(W\). It follows that all lines of \(\Lt\) through \(x_2\) are contained in \(\Lt_W\). That is, \(x_2\) is an ideal point of \(W\). Observe that also the unique point, which we will call \(x_3\), on \(L\) collinear to \(x_2\) is ideal in \(W\).
Again, there exists a line \(M\) at distance 5 from \(x_2\) in \(\Lt_W\).

By Corollary \ref{cor:UniqueClosest}, both \(x_1\) and \(x_3\) determine points \(x_1'\) and \(x_3'\) on \(M\) at closest distance. If \(\delta(x_1,x_1')=2\), then \(\delta(x_3,x_3')\neq 2\), otherwise we would find a pentagon in \(\T\). Hence, at least one of the ideal points $x_1,x_3$ is at distance $5$ from \(M\) and we may assume that without loss of generality \(\delta(x_1,x_1') = 4\). 
Let \(c_1 := x_1\bowtie x_1'\) and \(c_2 := x_2\bowtie x_2'\). Since \(x_1c_1\) and \(x_2c_2\) are in \(\Lt_W\) we have constructed a hexagon \(\mathscr{H} = \{x_1x_2,x_2c_2,c_2x_2',M,x_1'c_1,c_1x_1\}\) in \(\Lt_W\).

The subspace \(\langle \mathscr{H} \rangle\) is 5-dimensional as it is spanned by two planes, \(\langle x_1,x_2,c_2\rangle\) and \(\langle x_1',x_2',c_1 \rangle\), and no 4-dimensional spaces contain hexagons (Corollary \ref{cor:noHexIn4}).
As \(W\) is \(\Lt\)-supported, there exists a line \(N\) of \(\Lt_W\) not in \(\langle \mathscr{H} \rangle\). All lines intersecting two lines of the hexagon \(\mathscr{H}\) are contained in \(\langle \mathscr{H} \rangle\) so considering symmetry we may split into two cases:

Case 1: \(N\) intersects \(x_1x_2\) in a point \(y\) different from \(x_1\) and \(x_2\).

Case 2: \(N\) does not intersect any line of the hexagon.

We now prove that in both cases, \(W\) contains a heptagon. Lemma \ref{lem:hepToSC} then shows that \(W\) intersects \(\T\) in a split Cayley hexagon and Corollary \ref{cor:SCHBlocksInTT} implies that there can be no isolated points. 

\paragraph{Case 1: \(N\) intersects \(x_1x_2\) in a point \(y\) different from \(x_1\) and \(x_2\).}
Let \(y'\) be one of the points collinear with \(x_1'\), at distance \(4\) of \(y\) and not on \(c_1x_1'\) or \(x_1'x_2'\). Clearly \(y'\) is contained in \(W\) since all lines of \(\Lt\) through \(x_1'\) are contained in \(W\). 
Similarly, \(c_3 := y'\bowtie y\) is contained in \(W\) since \(N\) and \(x_1x_2\) through \(y\) are contained in \(W\).
We have constructed the following heptagon in \(W\):
\[\{x_1x_2,x_2c_2,c_2x_2',x_2'x_1',x_1'y',y'c_3,c_3y\}.\]

\begin{figure}[h]
	\begin{center}
		\resizebox{0.3\textwidth}{!}{
			\begin{tikzpicture}[outer sep=3, inner sep=3]
				\foreach \P in {
					(0:3) coordinate (p1) node[anchor=west]{\(c_1\)},
					(1*60:3) coordinate (p2) node[anchor=south]{\(x_1\)}, 
					(2*60:3) coordinate (p3) node[anchor=south]{\(x_2\)}, 
					(3*60:3) coordinate (p4) node[anchor=east]{\(c_2\)}, 
					(4*60:3) coordinate (p5) node[anchor=north]{\(x_2'\)}, 
					(5*60:3) coordinate (p6) node[anchor=north]{\(x_1'\)}, 
					($(p2)!0.5!(p3)$) coordinate (y) node[anchor=south]{\(y\)},
					(5*60:1) coordinate (yp) node[anchor=south east]{\(y'\)},
					(1,1) coordinate (c3) node[anchor=west]{\(c_3\)}
				}
				\filldraw \P circle (3pt);
				
				\draw[thick]
				(p1) -- (p2) -- (p3) -- (p4) -- (p5) -- (p6) -- cycle
				;
				\draw[thick]
				(p6) -- (yp) -- (c3) -- (y) --node[anchor=east]{\(N\)} (-1,0);
				;
				
			\end{tikzpicture}
		}
	\end{center}
	\vspace{-1em}
	\caption{Case 1.}	
\end{figure}

\paragraph{Case 2: \(N\) does not intersect any line of the hexagon.}
Without loss of generality, we may assume that \(c_1\) and \(c_2\) are at distance 5 from \(N\) since at most two of the points \(x_1,x_2,c_1,c_2,x_1'\) and \(x_2'\) can be at distance 3 from \(N\). 

By Corollary \ref{cor:UniqueClosest}, both \(c_1\) and \(c_2\) determine points \(y_1\) and \(y_2\) respectively, on \(N\) such that \(\delta(c_1,y_1) = 4 = \delta(c_2,y_2)\).

Let \(z_1 := x_1\bowtie y_1\) and \(z_2 := x_2\bowtie c_2\). Clearly \(z_1\) and \(z_2\) are contained in \(W\) by property (Pl) (Lemma \ref{lem:pl}). Now consider the unique point \(z_1'\) on \(c_1z_1\) at distance 4 from \(c_2\). We let \(v\) be the unique point collinear with \(c_2\) and \(z_1'\). Property (Pl) (Lemma \ref{lem:pl}) again implies that \(v\) is a point of \(W\).
If \(z_1'\neq z_1\) then we have constructed the following heptagon in \(W\) (and we are done):
\[\{c_2v, vz_1',z_1'z_1, z_1y_1,N,y_2z_2,z_2c_2\}.\]

\begin{figure}[h]
	\begin{center}
		\resizebox{0.3\textwidth}{!}{
			\begin{tikzpicture}[outer sep=3, inner sep=3]
				\foreach \P in {
					(0:3) coordinate (p1) node[anchor=west]{\(c_1\)},
					(1*60:3) coordinate (p2) node[anchor=south]{\(x_1\)}, 
					(2*60:3) coordinate (p3) node[anchor=south]{\(x_2\)}, 
					(3*60:3) coordinate (p4) node[anchor=east]{\(c_2\)}, 
					(4*60:3) coordinate (p5) node[anchor=north]{\(x_2'\)}, 
					(5*60:3) coordinate (p6) node[anchor=north]{\(x_1'\)}, 
					(1,0) coordinate (y1) node[anchor=west]{\(y_1\)},
					(-1,0) coordinate (y2) node[anchor=east]{\(y_2\)},
					(1,1.8) coordinate (z1) node[anchor=south]{\(z_1\)},
					(-1,-1.8) coordinate (z2) node[anchor=north]{\(z_2\)},
					(2,0.9) coordinate (z1p) node[anchor=north]{\(z_1'\)},
					(-2,0.9) coordinate (v) node[anchor=south]{\(v\)}
				}	
				\filldraw \P circle (3pt);
				
				\draw[thick]
				(p1) -- (p2) -- (p3) -- (p4) -- (p5) -- (p6) -- cycle
				;
				\draw[thick]
				(p1) -- (z1) -- (y1) -- node[anchor=south]{\(N\)} (y2) -- (z2) -- (p4)
				;
				\draw[thick]
				(z1p) -- (v) -- (p4)
				;
				
			\end{tikzpicture}
		}
	\end{center}
	\vspace{-1em}
	\caption{Case 2 if \(z_1'\neq z_1\).}	
\end{figure}
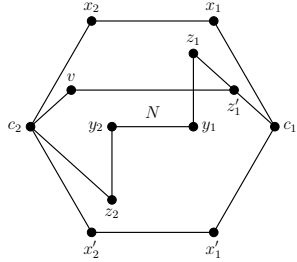

We may therefore assume that \(z_1'=z_1\).
By symmetry, we may also assume that \(z_2\) is the unique point on \(c_2z_2\) at distance 4 from \(c_1\) since otherwise we find a similar heptagon through \(c_1\) and \(z_2\).

Let \(v'\) be the unique point collinear with both \(c_1\) and \(z_2\). Both the lines \(c_2v\) and \(c_1v'\) are in \(\langle\mathscr{H} \rangle\) and belong to \(T(z_1,z_2)\). Hence, all lines of \(T(z_1,z_2) = R(c_2v,c_1v')\) are in \(\langle\mathscr{H} \rangle\) (Lemma \ref{lem:3supported}). Since \(N\) also belongs to \(T(z_1,z_2)\) this contradicts the assumption that \(N\) was not in \(\langle\mathscr{H} \rangle\).

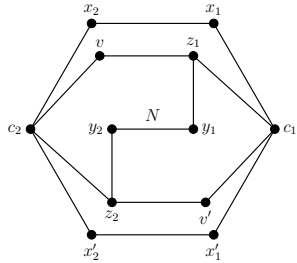
\begin{figure}[h]
	\begin{center}
		\resizebox{0.3\textwidth}{!}{
			\begin{tikzpicture}[outer sep=3, inner sep=3]
				\foreach \P in {
					(0:3) coordinate (p1) node[anchor=west]{\(c_1\)},
					(1*60:3) coordinate (p2) node[anchor=south]{\(x_1\)}, 
					(2*60:3) coordinate (p3) node[anchor=south]{\(x_2\)}, 
					(3*60:3) coordinate (p4) node[anchor=east]{\(c_2\)}, 
					(4*60:3) coordinate (p5) node[anchor=north]{\(x_2'\)}, 
					(5*60:3) coordinate (p6) node[anchor=north]{\(x_1'\)}, 
					(1,0) coordinate (y1) node[anchor=west]{\(y_1\)},
					(-1,0) coordinate (y2) node[anchor=east]{\(y_2\)},
					(1,1.8) coordinate (z1) node[anchor=south]{\(z_1\)},
					(-1,-1.8) coordinate (z2) node[anchor=north]{\(z_2\)},
					(-1.3,1.8) coordinate (v) node[anchor=south]{\(v\)},
					(1.3,-1.8) coordinate (vp) node[anchor=north]{\(v'\)}
				}	
				\filldraw \P circle (3pt);
				
				\draw[thick]
				(p1) -- (p2) -- (p3) -- (p4) -- (p5) -- (p6) -- cycle
				;
				\draw[thick]
				(p1) -- (z1) -- (y1) -- node[anchor=south]{\(N\)} (y2) -- (z2) -- (p4)
				;
				\draw[thick]
				(z1) -- (v) -- (p4)
				(z2) -- (vp) -- (p1)
				;
				
			\end{tikzpicture}
		}
	\end{center}
	\vspace{-1em}
	\caption{Case 2 if \(z_1' = z_1\) and \(\delta(z_2,c_1) = 4\).}	
\end{figure}

\end{proof}

\begin{corollary}\label{cor:6HexIsSC}
Let \(W\) be a \(6\)-dimensional \(\Lt\)-supported subspace of \(\PG\) which contains a hexagon.
Then, \(W\) intersects \(\T\) in a split Cayley subhexagon.
\end{corollary}

We now investigate the \(\Lt\)-supported subspaces containing only disjoint lines.
\begin{lemma}\label{lem:HermInSCH}
Let \(U\) be a 5-dimensional \(\Lt\)-supported subspace of \(\PG\) with the property that no two lines of \(\Lt(U)\) intersect. Then, there exists a 6-dimensional subspace \(V\) of \(\PG\) such that \(U\) is contained in \(V\) and \(V\) intersects \(\T\) in a split Cayley hexagon.
\end{lemma}
\begin{proof}
By Lemma \ref{lem:SupSubspace} and Lemma \ref{lem:3supported}, \(U\) contains three mutually opposite lines \(L_1,L_2\) and \(L_3\) such that \(\langle L_1,L_2,L_3 \rangle = U\). 
Let \(x_1\) be a point on \(L_1\) and let \(x_2\) be the unique point on \(L_2\) at distance \(4\) from \(x\).

Let \(y = x_1\bowtie x_2\). Clearly \(R(L_1,L_2)\) is contained in \(U\) and hence we know that \(L_3\) is opposite all lines of \(R(L_1,L_2)\). Lemma \ref{lem:LineOppositeRegulus} implies that \(L_3\) is at distance \(5\) from \(y\).
Let \(M\) be the unique line through \(y\) at distance 4 from \(L_3\) (Corollary \ref{cor:UniqueClosest}).

The subspace \(\langle L_1, M, L_3 \rangle =: W\) is clearly 5-dimensional and, since it contains a path of 5 lines, must intersect \(\T\) in a subhexagon \(\mathcal{H}\) of order \((1,q)\) (Corollary \ref{cor:pathIn5IsSub}).

The line \(L_2\) cannot be contained in \(W\) since otherwise \(U\) would equal \(W\) and contain intersecting lines of \(\Lt\). We do know that \(L_2\) intersects the line \(yx_2\) which, as a line through \(y\), is contained in \(W\) (property (Pl)). Therefore \(\langle W, L_2 \rangle =: V\) is a 6-dimensional \(\Lt\)-supported subspace. Since \(W\) contains the subhexagon \(\mathcal{H}\), Corollary \ref{cor:6HexIsSC} implies that \(V\) intersects \(\T\) in a split Cayley Hexagon.
\end{proof}

In any split Cayley hexagon one can construct what is called a \emph{hermitian spread}, which is a set of \(q^3+1\) mutually opposite lines, for more details see \cite{Thas1980}.

Thas and Van Maldeghem mention in \cite{ThasVanMaldeghem2008} that there are three possibilities for the intersection of a 5-dimensional subspace and a regularly embedded split Cayley hexagon. 
In particular, if \(|\Lt(W)| = q^3+1\) for some 5-dimensional subspace \(W\) of \(\mathsf{PG}(6,q)\), than this intersection is precisely a hermitian spread.
In this same paper, the authors also prove explicitly (Lemma 2) that the required properties which can be carried over imply the following:
Every 5-dimensional subspace of \(\mathsf{PG}(6,q)\) contains either \(q^3+1\), \(q^3+q^2+q+1\) or \((q+1)(q^2+q+1)\) lines of a regularly embedded split Cayley hexagon.

\begin{corollary}\label{cor:Opp5Max}
Let \(V\) be a 5-dimensional \(\Lt\)-supported subspace of \(\PG\). If no two lines of \(\Lt\) intersect then \(|\Lt(V)| = q^3+1\), all lines of \(\Lt(V)\) are opposite in \(T\) and these lines form a hermitian spread in a split Cayley subhexagon.
\end{corollary}
\begin{proof}
By Lemma \ref{lem:HermInSCH}, the subspace \(V\) is contained in a subspace \(W\) which intersects \(\T\) in a split Cayley subhexagon. The result now follows from Lemma \ref{lem:MaximumSetOfOppositeLines} and the description in \cite{ThasVanMaldeghem2008}.
\end{proof}

\begin{lemma}\label{lem:5supported}
\textbf{(5d)} Let \(V\) be a \(5\)-dimensional \(\Lt\)-supported subspace of \(\PG\).
Then, one of the following is true:
\begin{itemize}
	\item \(|\Lt(V)| = q^3+1\).
	More specifically, all lines of \(\Lt(V)\) are opposite in \(T\) and form a hermitian spread in a split Cayley subhexagon.
	There are no ideal points in this case.
	
	\item \(|\Lt(V)| = q^3+q^2+q+1\).
	
	More specifically, there is a point \(x\) in \(V\) such that all lines of \(\Lt(V)\) are contained in \(\Lt_{\le 3}(x)\) and each line through \(x\) intersects \(q^2+q\) of the other lines.
	All ideal points of \(V\) are on lines of \(\Lt(V)\) though \(x\) and therefore coplanar.
	Each of the lines through \(x\) contains precisely \(q+1\) ideal points of \(V\).
	
	\item \(|\Lt(V)| = q^3+2q^2+2q+1\).
	
	More specifically, all lines of \(\Lt(V)\) (together with the ideal points of \(V\)), form a subhexagon of order \((1,q)\).
	
	\item \(|\Lt(V)| = q^4+q+1\).
	
	More specifically, there exists a line \(M\) in \(\Lt(V)\) such that \(\Lt(V) = \Lt_{\le 2}(L)\).
	This means there are \(q^3+1\) ideal points in \(V\). They are precisely all the points on \(M\).
\end{itemize}
In any case, there are no isolated points in \(V\).
\end{lemma}
\begin{proof}
In view of Corollary \ref{cor:Opp5Max} and Lemma \ref{lem:5supportedInt}, the only thing left to prove is that in none of the cases the space \(V\) contains isolated points. 
This follows from extending each different case to a 6-dimensional \(\Lt\)-supported subspace containing them. 

We first consider the case where \(|\Lt(V)| = q^3+1\). We have already shown that in this case there exists a \(6\)-dimensional subspace \(W\) which contains \(V\) and intersects \(\T\) in a split Cayley hexagon (see Lemma \ref{lem:HermInSCH}). From Lemma \ref{lem:6supportedInt}, we get that \(W\) does not contain any isolated points. By Lemma \ref{lem:SpreadBlocksLines}, all lines of \(\Lt(W)\) intersect the \(q^3+1\) lines in \(\Lt(V)\). This proves the claim that \(V\) does not contain any isolated points. 

Next, consider the case where \(|\Lt(V)| = q^3+2q^2+2q+1\). If we consider any line \(N\) of \(T\) not in \(V\) which intersects one of the lines of the subhexagon of order \((1,q)\) in a point not of the subhexagon, then \(\langle N, V\rangle =: W\) must be an \(\Lt\)-supported 6-dimensional subspace intersecting \(\T\) in a split Cayley hexagon (Corollary \ref{cor:6HexIsSC}). From Lemma \ref{lem:6supportedInt}, we again get that \(W\) does not contain any isolated points. By Corollary \ref{cor:IdealSubhexagonOfSCH}, all lines of \(\Lt(W)\) intersect the lines in \(\Lt(V)\). This proves the claim that \(V\) does not contain any isolated points.

For the other two cases, \(|\Lt(V)| = q^3+q^2+q+1\) and \(|\Lt(V)| = q^4+q+1\), we can similarly use Lemma \ref{lem:pointTangent}(\ref{lem:pointTangent:noIsoSub}) and \ref{lem:lineTangent}(\ref{lem:lineTangent:noIsoSub}) respectively.
\end{proof}               

\begin{lemma}\label{lem:HermNoOpp}
Let \(V\) be a 5-dimensional \(\Lt\)-supported subspace of \(\PG\) with the property that no two lines of \(\Lt(V)\) intersect. Then, there are no lines in \(\Lt\) opposite all lines of \(\Lt(V)\).
\end{lemma}      
\begin{proof}
We have shown in Lemma \ref{lem:HermInSCH} that \(V\) is contained in a 6-dimensional subspace \(W\) which intersects \(\T\) in a split Cayley hexagon. By Lemma \ref{lem:SpreadBlocksLines} every line of this split Cayley hexagon intersects some line of \(\Lt(V)\). 
By Corollary \ref{cor:SCHBlocksInTT}, every line of \(\Lt\) intersects a line of \(\Lt(W)\). 
Combined, this implies that for any line of \(\Lt\) there is a line at distance 4 from it in \(\Lt(V)\).
\end{proof}

\begin{corollary}\label{cor:6supported}
\textbf{(6d)} Let \(W\) be a \(6\)-dimensional \(\Lt\)-supported subspace of \(\PG\).
Then, one of the following is true:
\begin{itemize}
	\item \(|\Lt(W)| = q^5+q^4+q+1\).
	More specifically, there is a point \(x\) in \(V\) such that \(\Lt(V) = \Lt_{\le 3}(x)\).
	The ideal points of \(W\) are precisely the points \(\Pt_{\le 2}(x)\).
	\item \(|\Lt(W)| = q^5+q^4+q^3+q^2+q+1\).
	More specifically, all lines of \(\Lt(W)\) (and the ideal points of \(W\)) form a split Cayley subhexagon. This split Cayley hexagon is regularly embedded in a subspace \(\mathsf{PG}(6,q)\) of \(W\). 
\end{itemize}
In either case, \(W\) does not contain any isolated points.
\end{corollary}
\begin{proof}
We use Lemma \ref{lem:SupSubspace}.
Either there is an \(\Lt\)-supported \(4\)-space and then Lemma \ref{lem:4supported} gives intersecting lines or there is an \(\Lt\)-supported \(5\)-space. If this \(\Lt\)-supported \(5\)-space does not contain intersecting lines, then Lemma \ref{lem:HermNoOpp} shows that \(W\) does contain non-opposite lines.

By Lemma \ref{lem:3supported}, the solid spanned by non-opposite non-intersecting lines contains two intersecting lines of \(\Lt\).

This shows that the condition of having two intersecting lines in Lemma \ref{lem:6supportedInt} is always satisfied. 
\end{proof}

\section{A line set \(\Lines\) satisfying (Pt), (Pl) and (Sd)}\label{sec:LineSet}

For this section we assume that \(\Lines\) is a set of lines in \(\PG\) which satisfies the properties (Pt), (Pl) and (Sd).

We first observe that the embedding is flat (the same argument appears in \cite[Lemma 1]{ThasVanMaldeghem2008}.)
\begin{lemma}\label{Lem:Flat}
	The \(q+1\) lines of \(\Lines\) through a fixed point \(x\) are coplanar.
\end{lemma}
\begin{proof}
	Suppose that there are three lines \(L_1\), \(L_2\) and \(L_3\) in \(\Lines\) which meet in a point \(x\) and such that \(\langle L_1, L_2,L_3\rangle\) is a solid. By (Pl), each of the planes \(\langle L_1,L_2\rangle\),  \(\langle L_1,L_3\rangle\) and  \(\langle L_2,L_3\rangle\) contains exactly \(q+1\) members of \(\Lines\). It follows that \(\langle L_1, L_2,L_3\rangle\) contains at least \(3q\) elements of \(\Lines\) which contradicts the property (Sd).
\end{proof}

\begin{lemma}\label{lem:noTriangles}
No three lines of \(\Lines\) form a triangle.
\end{lemma}
\begin{proof}
This follows directly from the embedding being flat (Lemma \ref{Lem:Flat}) and properties (Pl) and (Pt).
\end{proof}

\begin{lemma}\label{lem:NoQuad}
No four lines of \(\Lines\) form a quadrangle.
\end{lemma}
\begin{proof}
Assume that the lines \(L,M,N\) and \(O\) form a quadrangle. By property (Pl) the plane through \(L\) and \(M\) must have \(q+1\) lines. Similarly, the plane through \(M\) and \(N\) must have \(q+1\) lines. This gives us \(2q+1\) different lines of \(\Lines\) in the solid spanned by \(L\) and \(N\). By property (Sd) this solid can have no other lines which contradicts the existence of \(O\).
\end{proof}

\begin{lemma}\label{lem:ThreeLines}
Let \(L,M\) and \(N\) be three lines of \(\Lines\) such that \(M\) intersects both \(L\) and \(N\) but \(L\) and \(N\) do not intersect. Then, the solid \(\Sigma = \langle L,N \rangle\) contains \(2q+1\) lines of \(\Lines\). 
\end{lemma}
\begin{proof}
By property (Pl), the plane through \(L\) and \(M\) in \(\Sigma\) contains \(q+1\) lines in \(\Lines\) which by (Pt) and Lemma \ref{Lem:Flat}, all meet in the point \(x = L\cap M\). Similarly, the plane through \(M\) and \(N\) in \(\Sigma\) contains \(q+1\) lines in \(\Lines\) which all meet in the point \(y = M\cap N\).
Property (Sd) implies that \(\Sigma\) cannot contain any other lines of \(\Lines\).
\end{proof}

\begin{lemma}\label{lem:4dNoPenta}
If \(\Lines\) satisfies property (4d'), then no five lines form a pentagon.
\end{lemma}
\begin{proof}
Assume that there exists a pentagon \(\mathscr{P} := \{L_1,L_2,L_3,L_4,L_5\}\) in \(\Lines\). It follows from Lemma \ref{lem:noTriangles} that \(L_1,L_2\) and \(L_3\) span a solid \(\Sigma\).

We consider the solid through \(L_1,L_2\) and \(L_3\) and call this solid \(\Sigma\). 
If \(L_4\) or \(L_5\) were also in \(\Sigma\), we easily find a contradiction as in the proof of Lemma \ref{lem:NoQuad} so we may assume this is not the case. Let \(z\) be the intersection of \(L_4\) and \(L_5\). We have that \(z\not\in \Sigma\).
Therefore, the plane \(\pi\) though \(L_4\) and \(L_5\) meets \(\Sigma\) in a line \(M'\).
As this plane \(\pi\) already contains two lines in \(\Lines\), property (Pl) implies it must contain \(q+1\) lines in \(\Lines\). 
We denote this set of \(q+1\) lines by \(\Lines_\pi\). Note that all lines of \(\Lines_\pi\) contain \(z\) by (Pt) and Lemma \ref{Lem:Flat}.

We now consider all the solids spanned by \(L_2\) and a line in \(\Lines_\pi\).
Each two of these solids meet in the plane \(\pi'\) spanned by \(L_2\) and \(z\).
This plane \(\pi'\) only contains the line \(L_2\) of \(\Lines\) because it is in \(\langle L_2,L_4 \rangle\).

Clearly, each of these solids contains two, and thus, by (Sd), at least \(q+1\) lines of \(\Lines\).
For the solids generated by \(L_2\) and \(L_4\) and by \(L_2\) and \(L_5\) we can do even better, as Lemma \ref{lem:ThreeLines} implies that these solids must contain at least \(2q+1\) lines of \(\Lines\).
We find the following lower bound for the number of lines of \(\Lines\) in the 4-dimensional space spanned by \(\Sigma\) and \(z\):
\[\sum_{M \in \Lines_\pi} |(\Lines(\langle L_2, M \rangle)| - 1) + 1 \ge 2\cdot2q + (q-1)\cdot q + 1 = q^2+3q+1.\]

This contradicts property (4d').
\end{proof}

\section{Conclusion}

\begin{proof}[Proof of Theorem \ref{thm:Main1:SatisfiesProperties}]
	In Section \ref{sec:EmbeddedTProp}, we have shown that the lines \(\Lt\) of a naturally embedded twisted triality hexagon \(\TT\) in \(\PG\) satisfy the properties (Pt), (Pl), (Sd), (4d), (4d'), (5d), (6d) and (To).
\end{proof}

\begin{proof}[Proof of Theorem \ref{thm:Main1.5:4dIsTwistedTriality}]
	In Section \ref{sec:LineSet}, we have shown that any set of lines satisfying (Pt), (Pl) and (Sd) contain no triangles and no quadrangles.
	In Lemma \ref{lem:4dNoPenta}, we have shown that if this set also satisfies property (4d') we can have no pentagons.
	Using property (Pt) and the fact that the lines are full, we now count the number of lines in \(\Lines\).
	Starting from a point \(x\), there are \(q+1\) lines of \(\Lines\) through \(x\). Since there are no triangles, we find through each of the points on these lines, \(q\) more lines of \(\Lines\). This gives us \((q+1)q^4\) lines of \(\Lines\) at distance \(3\) from \(x\). Since there can be no quadrangles, we can also easily count the number of points on lines of \(\Lines\) at distance \(4\) from \(x\) to be \((q+1)q^7\). Finally since there are no pentagons, the number of lines of \(\Lines\) at distance \(5\) from \(x\) must be \((q+1)q^8\).
	This gives us a lower bound of 
	\[(q+1) + (q+1)q^4 + (q+1)q^8 = q^9+q^8+q^5+q^4+q+1\]
	for the number of lines in \(\Lines\).
	Property (To) implies we precisely meet this bound. It follows that all lines of \(\Lines\) are at distance at most \(5\) from \(x\). Since \(x\) was arbitrary, it follows that the maximal distance between any line of \(\Lines\) and any point on a line of \(\Lines\) is at most \(5\). 
	
	This implies that the maximum length of a path between two vertices in the incidence graph is \(6\). 
	Furthermore, by considering a point \(x\) and a line \(L\) at distance 5, we find a path \(x,L_1,y,L_2,z,L\). Let \(M\) be a line through \(x\), different from \(L_1\) and let \(x'\) be a point on \(M\) different from \(x\). Then, there is a path \(x',L_1',y',L_2',z',L\). Since there are no cycles of length 10 in this graph, the cycle \(x,L_1,y,L_2,z,L,z',L_2',y',L_1',x',M\) has length 12. Since we have an order \((q^3,q)\), Corollary \ref{cor:incidenceGraph} shows that it is a generalised hexagon.
	The order \((q^3,q)\) also indicates that this generalised hexagon is thick (see Remark \ref{rem:ThickGeneralisedHexagon}).
	
	This generalised hexagon is clearly fully embedded, since we started from a set of lines of \(\PG\), and is flatly embedded (Lemma \ref{Lem:Flat}). 

	We now use Theorem \ref{thm:KeyTheorem}(2).
	
	We conclude that the embedded hexagon is a regular embedding of a twisted triality hexagon \(\TT\) in \(\PG\).
\end{proof}

\paragraph{Acknowledgement}
The authors would like to thank the anonymous reviewer for their helpful suggestions.

\end{document}